%% file: ERMLR_preprint.tex
\documentclass[twoside,11pt]{article}

\usepackage{epsfig, amssymb, natbib, graphicx}

\usepackage{graphicx}
\usepackage{booktabs} 
\usepackage{hyperref} 

\usepackage{algorithm}
\usepackage{algorithmic}
\usepackage{tabularx}
\usepackage{amsmath, amsfonts, amsthm, amssymb, here, dsfont, hyperref}
\usepackage{graphicx,color,multirow, here}
\usepackage{mathtools}
\usepackage[utf8]{inputenc} 
\usepackage[T1]{fontenc}
\usepackage[english]{babel}
\usepackage{here,dsfont, hyperref}
\usepackage{natbib}
\usepackage{enumitem}
\usepackage{url}
\usepackage{float}
\usepackage[toc,page,title,titletoc,header]{appendix} 
\newcommand{\argmax}[1]{\underset{#1}{\operatorname{arg}\!\operatorname{max}}\;}
\newcommand{\argmin}[1]{\underset{#1}{\operatorname{arg}\!\operatorname{min}}\;}

\newcommand{\w}{\widehat}

\newcommand{\one}{\mathds{1}}
\newcommand{\cA}{\mathcal{A}}

\newcommand{\cD}{\mathcal{D}}
\newcommand{\E}{\mathbb{E}}

\newcommand{\cH}{\mathcal{H}}
\newcommand{\cM}{\mathcal{M}}

\renewcommand{\P}{\mathbb{P}}
\newcommand{\cR}{\mathcal{R}}
\newcommand{\R}{\mathbb{R}}

\newcommand{\cT}{\mathcal{T}}
\newcommand{\cY}{\mathcal{Y}}

\newcommand{\dd}{\ensuremath{\,\textrm{d}}}

\newtheorem{theo}{Theorem}
\newtheorem{prop}{Proposition}
\newtheorem{lemme}{Lemma}
\newtheorem{rem}{Remark}
\newtheorem{ass}{Assumption}

\usepackage{comment}
\usepackage[textwidth=1.8cm, textsize=tiny]{todonotes}
\usepackage{csquotes}
\usepackage{pgf}
\usepackage{subcaption}
\def\keywordname{{\bfseries Keywords}}%
\def\keywords#1{\par\addvspace\medskipamount{\rightskip=0pt plus1cm
\def\and{\ifhmode\unskip\nobreak\fi\ $\cdot$
}\noindent\keywordname\enspace\ignorespaces#1\par}}

\begin{document}

\title{ERM-Lasso classification algorithm \\ for Multivariate Hawkes Processes paths}
\author{}
\date{July 2024}

\maketitle

\begin{center}
\author{Christophe Denis$^{(1,2)}$, Charlotte Dion-Blanc$^{(1)}$, Romain E.\,Lacoste$^{(2)}$, Laure Sansonnet$^{(1,3)}$ \bigskip \\
\textit{${(1)}$ LPSM, UMR 8001, Sorbonne Université \\
${(2)}$ LAMA, UMR 8050, Université Gustave Eiffel \\
${(3)}$ Université Paris-Saclay, AgroParisTech, INRAE, UMR MIA Paris-Saclay}}
\end{center}

\bigskip

\begin{abstract}
We are interested in the problem of classifying Multivariate Hawkes Processes (MHP) paths coming from several classes. MHP form a versatile family of point processes that models interactions between connected individuals within a network. In this paper, the classes are discriminated by the exogenous intensity vector and the adjacency matrix, which encodes the strength of the interactions. The observed learning data consist of labeled repeated and independent paths on a fixed time interval. Besides, we consider the high-dimensional setting, meaning the dimension of the network may be large {\it w.r.t.} the number of observations. We consequently require a sparsity assumption on the adjacency matrix. In this context, we propose a novel methodology with an initial interaction recovery step, by class, followed by a refitting step based on a suitable classification criterion. To recover the support of the adjacency matrix, a Lasso-type estimator is proposed, for which we establish rates of convergence. Then, leveraging the estimated support, we build a classification procedure based on the minimization of a $L_2$-risk. Notably, rates of convergence of our classification procedure are provided. An in-depth testing phase using synthetic data supports both theoretical results. 
\end{abstract}


\section{Introduction}
\label{sec:intro}
\input{sec1_intro}

\section{General framework}
\label{sec:framework}
\input{sec2_framework}

\section{Classification algorithm}
\label{sec:classif}
\input{sec3_classif}

\section{Theoretical results}
\label{sec:theo}
\input{sec4_theorems}

\section{Implementation}
\label{sec:implem}
\input{sec5_implementation}

\section{Numerical results}
\label{sec:NumExp}
\input{sec6_num}

\section{Discussion}
\label{sec:Dicussion}
\input{sec7_discussion}

\subsection*{Acknowledgements}
This work has been supported by the Chaire \enquote{Modélisation Mathématique et Biodiversité} of Veolia-École polytechnique-Museum national d’Histoire naturelle-Fondation X, through a Ph.D. scholarship.
The project is also part of the 2022 DAE 103 EMERGENCE(S) - PROCECO project supported by Ville de Paris.
Finally, the authors thank Vincent Rivoirard for fruitful discussions.

\vskip 0.2in

\begin{appendices}
\input{sec8_appendix}

\end{appendices}

\vskip 0.2in
\bibliographystyle{apalike}
\bibliography{biblio}

\end{document}

%% file: sec1_intro.tex

The supervised classification of complex data has drawn a lot of attention in recent years. This statistical problem covers a broad class of application including classification of multivariate times series \citep{DLtimeseries}. In particular, cutting-edge methodology for performing classification of time sequences of events is a matter of great interest.
In this paper, we tackle the task of supervised classification of sequences of events into $K$ classes with $K>1$. We therefore consider that each class is characterized by its own underlying occurrence dynamics, and the aim is to discriminate between them. 
In this work, observations in each class are assumed to come from a multivariate Hawkes process (MHP), denoted $N$, of size $M\geq1$, for which the probability distribution of events 
is given by the vector of intensity process.
The shape of the vector of intensity process is assumed to be common to all classes. Therefore, the classes are discriminated according to the parameters that describe their vector of intensity process:
the baselines and the adjacency matrix that governs the relations between the components of the process.
For instance, we can consider that the observations come from different groups of people observed over a given time interval whose interactions are modeled through a MHP. 


In this work, the observations consist in $n$ independent repeated observations of a mixture of the MHP observed on a fixed time interval $[0,T]$ and their associated label. In particular, we do not assume that data have reached the stationary regime. Hence, the asymptotic is in $n$ the number of repetitions.
Furthermore, we consider the high-dimensional framework,
where the dimension $M:=M_n$ of the MHP may be large with respect to the sample size $n$.
In view of the high-dimensional issue, we consider sparsity assumption on the adjacency matrix of the process. We propose a classification procedure that take advantage of the estimation of the support of the adjacency matrix.


\paragraph{Related works.}

Hawkes processes (HP) are a family of point processes introduced by \cite{hawkes1971}. Such processes model complex temporal dynamics, where the occurrence of events is impacted by past activity.
The multidimensional version of these processes, MHP, is a natural generalization that considerably enriches modeling possibilities. 
Indeed, in addition to model self-exciting interactions, such a model takes into account positive interactions between connected individuals within a network. These interactions are encoded in the adjacency matrix. 
Historically applied in seismology \citep{ogata}, they have since been used in a wide range of applications including genomics \citep{reynaudschbath}, neuroscience
\citep{bonnet2022}, finance \citep{embrechts2011}, urban crime \citep{crime}, order book in finance \citep{finance} and football \citep{footmathieu}.
Another important application is the modeling of social network activity, as in \cite{twitter} and further works. In addition, the Hawkes processes are frequently used as spike-trains models in neurosciences, for example in \cite{reynaud2013spike,stefano2023, bonnet2023inference}.
Recently, they have also been used in the ecological field in \cite{ecoHawkes2024} for interaction between species and in \cite{DDBLSB} for bats monitoring.
Recently, in the context of repeated observations with $T$ fixed, \cite{lotz2024} provides a likelihood ratio test for testing presence of interaction.


In the high-dimensional setting, meaning that the number of components $M$ is large, it is classical to impose sparsity assumptions on the adjacency matrix that characterizes the intensity process.
Therefore, it appears that reconstructing the support of the adjacency matrix or the
connectivity graph, is a matter of great interest, 
which is related to the Granger causality \citep[see][]{EDD2017granger,sulem2024bayesian}.
In particular, this is a crucial issue for connectivity of neurons, see for example \cite{lambert}, or in social network \citep{hensen2010}. 


Let us focus on the Lasso literature in the classical Gaussian framework. The Lasso procedure has been originally introduced in \cite{tibshirani1996} and \cite{chen98lasso}. It is a popular statistical method for high-dimensional problems for which efficient implementation procedure have been developed. Besides, the Lasso procedure has been widely studied from the theoretical point of view~\citep[see \emph{e.g}][]{meinshausen2006high,tropp2006just,VDG}.
In particular, support recovery results has been investigated in~\cite{wainwright2009sharp} as well as multi-class classification methods~\citep[see][]{abramovich2018high}. Let us emphasize that an induced and undesired effect of $\ell_1$ penalization is the shrinkage of large coefficients. 
To bypass this issue, refitting strategies are commonly used and well covered in the literature, see for example \cite{Chzhen_2019}. 


For Hawkes process, the work of~\cite{donnet2020nonparametric} is dedicated to a nonparametric Bayesian procedure to tackle high-dimensionality.
Let us mention the work \cite{ADM4} which proposes an efficient algorithm for a Lasso estimator for high-dimension Hawkes process, implemented in the Python library \texttt{tick} \citep{tick}.
Later, in the work of~\cite{BGM}, the authors use a least-squares contrast penalized with an $\ell_1$-norm together with a trace norm. The resulting estimator benefits from a sparse structure with low rank. Its construction relies only on an observation over a time interval $[0,T]$, with $T \rightarrow +\infty$. 
In particular, theoretical results are obtained for the intensity process estimation under the asymptotic $T \rightarrow +\infty$.
Under the same observational setup,
several sparse support recovery procedure has also been proposed in the literature, for example a likelihood ratio testing procedure in \cite{kim}, or a group-Lasso least-squares penalized estimator in \cite{groupLasso2024}.


The present work falls in the supervised classification setting. In particular, we assume that the learning sample of size $n$ consists of i.i.d.\,labeled data where the features are the jump times  of a multivariate Hawkes process observed on the fixed time interval $[0,T]$. 
In a different observational setup, some works in natural language processing also tackle some similar issues as \cite{classiftwitter} and later \cite{tondulkar2022hawkes}. Closest to ours, the work of~\cite{DDS} provides a classification procedure for observations coming from a univariate HP where the classes are discriminated by the kernel of the intensity process. In particular, this method is applied in~\cite{DDBLSB}, 
for modeling echolocation calls of bats recording in several sites throughout France.
In the present work, we generalize the approach developed in~\cite{DDS} in the multivariate setting (MHP) under sparsity assumptions. 


\paragraph{Main contributions.}

In the present paper, we propose a novel classification algorithm, the \texttt{ERMLR} algorithm that relies on a two-step procedure. A first step is dedicated to the estimation of the support of the adjacency matrix as well as the weights of the mixture. Then, in a second step, taking advantage of the estimated support, we build a classifier based on the empirical risk minimization principle. We establish rates of convergence for both support estimator and classification procedure. Furthermore, we show through a numerical study that our algorithm exhibits good numerical properties. 
To sum up our contributions are threefold.

\begin{itemize}
\item First, we provide a general device to handle high-dimensional issue for MHP.  Following~\cite{BGM}, the estimation of the parameters of the process as well as the support of the adjacency matrix is based on the  minimization of a Lasso-penalized contrast.
We establish rates of convergence of the estimated support and the estimated coefficients of the MHP. Notably, our theoretical findings show that the established rates of convergence are comparable to those obtained in the classical Gaussian setting.  In particular, we extend the results obtained in~\cite{BGM} and~\cite{groupLasso2024} in the context of repeated observations.
\item Second, we provide a general classification algorithm dedicated to the supervised classification of MHP. A salient point of our procedure is that we handle high-dimensional issue by leveraging the estimation of the support of the adjacency matrix. Specifically, we consider a classifier that relies on the minimization of a $L_2$-risk
on a set of parameters that depends only on the estimated support of the parameters. We show that the rates of convergence of our classification algorithm is, up to a logarithmic factor, of order the square root of the size of the support over the sample size $n$.
Notably, we extend the results obtained in~\cite{DDS} to the 
high-dimensional framework.
\item Finally, in view of the numerical complexity of our problem, the implementation of our classification algorithm is a major challenge. 
The implementation of the overall procedure relies on cutting-edge optimization algorithms. Specifically, the Lasso-penalized contrast is optimized with the \texttt{FISTA} algorithm while the calibration of the $l_1$ penalty is performed using the \texttt{EBIC} criterion. Then, for the minimization of the $L_2$-risk, we consider a parameter-free projected adaptive gradient descent \texttt{Free Adagrad} recently introduced in~\cite{chzhen2023}. We evaluate the performance of our procedure on synthetic data and show the good performance of our algorithm. In particular, it reveals that our algorithm succeeds well for both support recovery and classification accuracy.
\end{itemize}


\paragraph{Outline of the paper.}

Section~\ref{sec:framework} describes the model, along with the necessary assumptions.
Section~\ref{sec:classif} proposes the classification algorithm named \texttt{ERMLR}.
Section~\ref{sec:theo} provides the main theoretical results on both Lasso estimator of the MHP parameters and the classification procedure.
Then, full implementation details about the procedure are given in Section~\ref{sec:implem} while Section~\ref{sec:NumExp} is devoted to numerical results. We also provide a discussion in Section \ref{sec:Dicussion}. Finally, the proofs are relegated in Appendix.


\paragraph{Notations.}
For a matrix $A \in \mathbb{R}^{M \times M}$, the Frobenius norm is defined as follows $\|A\|_F = (\textrm{Tr}(A^\prime A))^{1/2} = (\sum_{j,j'} a_{j,j'}^2)^{1/2}$, where $A^\prime$ denotes the transposition of the matrix $A$ and $\textrm{Tr}$ is the trace operator that returns the sum of diagonal entries of a square matrix. 
Recall that $\rho(A) \leq \|A\|_2 \leq \|A\|_F$ where $\|A\|_2$ is the subordinate norm and $\rho(A)$ is the spectral radius of $A$, which is the largest singular value of $A$.
For an integer $L \geq 1$, the set $\{1, \ldots,L\}$ is denoted by $[L]$.


%% file: sec2_framework.tex

Section~\ref{subsec:modelNotations} introduces the considered model, some notation, and the considered multiclass classification problem.
Section~\ref{subsec:Assumptions} is dedicated to the presentation of the main assumptions. Finally,
a closed-form expression of the optimal classifier is provided in Section~\ref{subsec:BayesClassifier}.

\subsection{Formal definitions and notation}
\label{subsec:modelNotations}

Let us first introduce the general linear multivariate Hawkes process and then the considered multiclass classification model.

\paragraph{Multivariate counting process.}

Consider a $M$-dimensional counting process $N$ observed on a fixed time interval $[0,T]$, with $M>1$ the dimension of the network.
More specifically, we assume that the counting process $N=(N_1(t),\ldots,N_M(t))_{t \in [0,T]}$ is a linear multivariate Hawkes process, where for each $j \in [M]$, $t \in [0,T]$, $N_{j}(t)$ denotes the number of events that have occurred before time $t$ at location $j$.
The filtration (or history) at time $t \in [0,T]$ associated to the process $N$ 
is denoted by $\mathcal{F}_{t}$. Informally, it contains the necessary information for generating the next points of $N$. Finally, the set of observed jump times of $N$ over $[0,T]$ is denoted by $\mathcal{T}_T = (\mathcal{T}_{T,1}, \ldots, \mathcal{T}_{T,M})$, where for each $j \in [M]$, $\mathcal{T}_{T,j}$ is the observed jump times associated to the process $N_j$. Each process $N_j$ can be characterized by its intensity function. Heuristically, at a given time, the intensity function gives the infinitesimal probability of observing an event in the near future, conditionally on the past of the process. For each $j \in [M]$, the predictable intensity of the process $N_j$ is then defined by 
\begin{equation}\label{def:int}
    \lambda^*_j(t)
    = \mu^*_j + \sum_{j'=1}^M a^*_{j,j'} \int_0^t h(t-s)\dd N_{j'}(s)
    = \mu^*_j + \sum_{j'=1}^M a^*_{j,j'} \sum_{T_\ell \in \cT_{t,j'}} h(t-T_\ell),
\end{equation}
where $\mu^* = (\mu^*_j)_{j \in [M]}$ is the vector of exogenous intensities, $A^* = (a^*_{j,j'})_{1 \leq j,j' \leq M}$ is the matrix of interactions, and $h$ is the kernel function. For each $j \in [M]$, the coefficient $\mu^*_j$ models the arrival of spontaneous events for the $j$-th component. For each $j,j' \in [M]$, the coefficient $a^*_{j,j'}$ is non-negative and expresses the positive influence of the one-dimensional process $N_{j'}$ on the one-dimensional process $N_j$.
Finally, the kernel $h$ is a non-negative
function supported on $\R_{+}$ such that $\|h\|_1 =1$. It dictates how quick these influences vanish over time.
In the following, the kernel function $h$ is assumed to be known.
Finally, let us define the support of $A^*$,
or the active set, denoted $S^*$. It corresponds to the positions of the non-zero coefficients $a_{j,j'}$, meaning that component $j'$ has an impact on component $j$.

\begin{rem}
The second equality in the definition of the intensity~\eqref{def:int} is easily obtained by using that, for any function $f$, $j \in [M]$, and $t \in [0,T]$, the following stochastic integral is defined as the counting measure
\begin{equation*}
    \int_0^t f(s) \dd N_j(s) = \sum_{T_\ell \in \cT_{t,j}} f(T_\ell).
    \end{equation*}
\end{rem}

\paragraph{Multiclass setting.}

We consider the multiclass classification problem, 
where each data point is characterized by a couple $(\mathcal{T}_T,Y)$, where $\mathcal{T}_T$ is the set of observed jump times of a counting process $N$ over $[0,T]$ and $Y \in [K]$ is its label. In particular, we assume that $N=(N_1,\ldots,N_M)$ is a mixture of a $M$-dimensional linear HP observed on the time interval $[0,T]$. More precisely, conditional on $Y$, the counting process $N$ is a $M$-dimensional linear HP, where for each $j \in [M]$, the predictable intensity 
of $N_j$ depends on the label $Y$ and is defined at time $t \geq 0$ as follows
\begin{equation}
\label{def:lambdastar}
\lambda^{*}_{Y, j}(t) = \mu^*_{Y,j}+ \sum_{j'=1}^M a^*_{Y,j,j'} \int_0^t h(t-s)\dd N_{j'}(s).
\end{equation}
The vector $\mu_Y^* = (\mu_{Y,j}^*)_{j \in [M]}$ is the vector of baselines associated to the class $Y$, and the matrix $A_Y^*=(a^*_{Y,j,j'})_{1 \leq j,j' \leq M}$ is the $M\times M$ adjacency matrix of the network associated to the label $Y$. 
This choice of modeling is motivated by the fact that the classes are characterized by different underlying network behavior, where an edge in the network matches a non-zero $a_{j,j'}$.

We assume in the following that the parameters $(\mu^*_{Y}, A^*_Y)$ are unknown as well as the distribution of $Y$ which is denoted by $p^* = (p^*_k)_{k \in [K]}$.
Finally, the kernel function $h$ is assumed to be known and for the sake of simplicity, it does not depend on the classes or on the components of the process. Note that in the numerical section, we consider the standard choice of exponential kernel. However, more general choice of the kernel function may be investigated. For instance, in~\cite{BGM} the authors consider the  case where $h$ is a sum of exponential functions, which preserves the markovianity of the intensity process. 

\paragraph{Objective.}

In the multiclass setup, the objective is to build a classifier, a measurable function $g$ such that
$g(\cT_T)$ belongs to $[K]$, and provides an accurate prediction of the label $Y$.
In particular, the misclassification risk assesses the quality of such predictor $g$. It is defined as 
\begin{equation*}
\cR(g) := \mathbb{P}\left(g(\cT_T) \neq Y\right).
\end{equation*}  
The set of all classifiers is denoted by $\mathcal{G}$. Naturally, we aim at considering the predictor $g^*$, namely the Bayes classifier, that achieves the minimum risk over $\mathcal{G}$. In Section~\ref{subsec:BayesClassifier}, we provide an explicit formula of the oracle classifier $g^*$. Nevertheless, since the distribution of the observation $(\mathcal{T}_T,Y)$ is assumed to be unknown, we build a predictor that relies on a training sample of size $n$ which consists of i.i.d. copies of $(\mathcal{T}_T, Y)$.
At this step, we draw the reader attention to the fact that the considered asymptotic is as follows. The horizon time $T$ is fixed, while the sample size $n$ goes to infinity. Recall that the size $M$ of the MHP is actually $M=M_n$ and can increase with $n$.
In the sequel, a predictor built on the training data is denoted by $\hat{g}$. In particular, we require that $\hat{g}$ satisfies
the consistency property,
\begin{equation*}
\E\left[\cR(\hat{g}) -\cR(g^*) \right] \rightarrow 0,    
\end{equation*}
when $n$ tends to infinity.
However, in our study, the intrinsic dimension of our problem $M^2$ may be much larger than the sample size of the learning dataset. In this case, predictor $\hat{g}$ is not consistent. Therefore, as it is usual in this high-dimensional setup, we introduce a sparsity assumption for our model.

\subsection{Assumptions}
\label{subsec:Assumptions}

This section is dedicated to the main assumptions that are assumed throughout the paper. In particular, in our multivariate framework, we allow the dimension parameter $M$ to be large, which may induces that $M^2$ is much larger than the size of the training sample. To alleviate this issue, we introduce a sparsity assumption on the matrices $(A_k^*)_{k \in [K]}$. 

Firstly, we introduce an assumption which ensures that each class occurs with non-zero probability. 

\begin{ass}
\label{ass:prob}
There exists $p_0 > 0$ such that $\min_{k \in [K]}({p}_k^*) > p_0$. 
\end{ass}

We also assume that the parameters of the process belongs to a compact set. 

\begin{ass} (Compactness)
\label{ass:mu}   
\begin{enumerate}
\item[(i)] There exists $0 < \mu_0 < \mu_1$ , {\it s.t.} for $i\in [M]$ and $k \in [K]$,
 $\mu_0 \leq \mu^*_{k,i}  \leq \mu_1$. 
 \item[(ii)] There exists $C_A>0$, {\it s.t.}  $\max_{k\in [K]} \|A_k^*\|_F< C_A $.
\end{enumerate}
\end{ass}

Furthermore, we consider the following assumptions, which imply that the process $N$ admits finite exponential moment.

\begin{ass} (Stability condition)
\label{ass:h}
\begin{enumerate}
\item[(i)] The kernel function $h$ belongs to the set $\cH:= \{ h: \R_+ \rightarrow \R_+, \int h(t) \dd t = 1 \}$ and  is bounded.
\item[(ii)] $\max_{k \in [K]} \rho(A_k^*) < 1$.
\end{enumerate}
\end{ass}

Let us notice here that if $C_A<1$ it implies that $\rho(A)<1$.

\begin{ass} (Exponential moment)
\label{ass:expomoments}
There exist $a>0$, and  $C>0$ that do not depend on $M$, such that
\begin{equation*}
    \sup_{j \in [M]} \E[\exp(aN_j(T))] < C.
\end{equation*}
\end{ass}


\begin{rem}
Note that~\cite{leblanc2024exponential} proves that the exponential moment of the multivariate Hawkes process is finite, under Assumption \ref{ass:h} and when the intensity process is stationary. Nevertheless, in the general case, the bound of the moment depends on $M$. Hence, we require a stronger condition, such as Assumption~\ref{ass:expomoments} which is more suitable in the high-dimensional framework. For instance, this assumption is satisfied if there exists a positive constant $C$, that does depend on $M$, such that $\sum_{j \in [M]} \mu_j \leq C$. 
\end{rem}

Finally, we assume that for each $k \in [K]$, the adjacency matrix $A_k^*$ is sparse, meaning that a few coefficients are non-zero. For each $k \in [K]$, let us denote by 
\begin{equation*}
S_k^* := 
\left\{(j,j') \in [M]^2 , \;\; a^*_{k,j,j'} \neq 0 \right\}
\end{equation*}
the active set (or support) of $A_k^*$, $|S_k^*|$ its cardinality, and $S_k^{*c}$ its complement. Throughout the paper, we consider the following assumption.

\begin{ass} (Sparsity assumption)
\label{ass:sparAss}
There exists a constant $s^* > 0$ such that 
$$\max_{k \in [K]} |S_k^*| \leq s^*.$$
\end{ass}

In particular, in our high-dimensional setting, we assume that $s^*<<M^2$.
Since we do not assume sparsity on the vectors $\mu_k, \ k \in [K]$, we consider the following interplay between parameter $M$ and the sample size of the training dataset.
The dimension of the process $M$ may depend on $n$ with $M^2 >> n$. In this case, the sparsity assumption is crucial to overcome the high-dimension issues. However, we assume that $M$ satisfies
$M/n \rightarrow 0$.

\begin{rem}
Note that we only assume sparsity on the adjacency matrix, but not on the vector $\mu^*$. It ensures that all the components of the process are active. However,
as in~\cite{BGM}, it may be possible to consider also sparsity assumption on the vector of exogenous intensities. Nevertheless, this is not the line taken in this work.
\end{rem}

\subsection{Bayes rule}
\label{subsec:BayesClassifier}

In this section, we exhibit a closed-form expression of the Bayes classifier $g^*$ that minimizes the misclassification risk over the set $\mathcal{G}$. 
The Bayes classifier is characterized by, 
\begin{equation*}
g^{*} \left(\mathcal{T}_T\right) \in \argmax{k \in \mathcal{Y}} {\pi^*_{k}(\mathcal{T}_T)},
\end{equation*}
with $\pi^*_k\left(\mathcal{T}_T\right)= \mathbb{P}\left(Y = k | \mathcal{T}_T\right)$.
The following result is an extension of the result given in \cite{DDS}. It gives the expression of the conditional probabilities $\pi^*_k$ and then provides a closed form of the Bayes classifier.

\begin{prop}
\label{prop:prop1}
Let $T \geq 0$. For each $k \in [K]$,
we define,  
\begin{equation}\label{def:F}
F^*_k(\cT_T)  
 := -\sum_{j=1}^M  \int_0^T \lambda^*_{k,j}(s) d s 
 + \sum_{j=1}^M \sum_{T_\ell\in \mathcal{T}_{T,j}} \log\left(\lambda^*_{k,j}(T_\ell )\right).
 \end{equation}
Therefore, the sequence of conditional probabilities satisfies
\begin{equation*}
\pi^*_k\left(\mathcal{T}_T\right) = \frac{ p^*_k{\rm e}^{{F}^*_k(\mathcal{T}_T)}}{\sum_{k'=1}^K p^*_{k'} {\rm e}^{{F}^*_{k'}(\mathcal{T}_T)}}\quad {\mathbb P}-a.s.,
\end{equation*}
where 
${F}^*=(F^*_1, \ldots, F^*_K)$.
\end{prop}

Proposition~\ref{prop:prop1} exhibits an explicit link between the unknown parameters $(\mu_k^*, A_k^*)_{k \in [K]}$ and the Bayes classifier. In particular, it suggests that 
a classification rule can be easily obtained by replacing the unknown parameters by estimators in Equation~\eqref{def:F}. However, the performance of the resulting classifier strongly depends on the quality of the considered estimators. In the present framework, without taking account Assumption~\ref{ass:sparAss}, the high-dimension of the problem could lead to bad estimates. To overcome this difficulty, we propose a classification algorithm tailored to our setting, which involves Lasso-type estimators.


%% file: sec3_classif.tex

In this section, we present the proposed classification algorithm that relies on a {\it refitting} strategy \citep[see \emph{e.g.}][]{Chzhen_2019}. The algorithm is referred as \texttt{ERMLR} for {\it Empirical Risk Minimizer with Lasso Refitting}.
Since the construction of the prediction rule goes in several steps and involves a splitting of the training dataset, for the sake of the simplicity, we consider a dataset of size $2n$. More specifically, the learning dataset is denoted 
denoted $\mathcal{D}_n = \{(\mathcal{T}^{(i)}_T, Y^{(i)}), i = 1, \ldots,2n\}$, which consists of $2n$ independent copies of  $(\cT_T, Y)$. For the estimation purpose, the data set $\mathcal{D}_n$ is divided into two independent data sets $\mathcal{D}_n^{(1)}$ and $\mathcal{D}_n^{(2)}$ of same size $n$. For sake of simplicity in the following, we index both sample using $\{1, \ldots, n\}$.


To take advantage of Assumption~\ref{ass:sparAss}, we then consider the following three-stages procedure.
\begin{itemize}
\item Based on the first data set $\mathcal{D}^{(1)}_n$, we estimate the distribution $p^*=(p_k^*)_{k \in [K]}$ by its empirical counterpart 
$\w{p}$. 
\item Based on the second dataset $\mathcal{D}_n^{(2)}$, and for each $k \in [K]$, we estimate by $\w{S}_k$ the active set $S_k^*$ with a Lasso-type criterion, described in Section~\ref{subsec:estimationSupport}.
\item Based on the second dataset $\mathcal{D}_n^{(2)}$, then, we build a classifier $\w{g}$ that minimizes an empirical $L_2$-risk on a set of predictors that depends on the estimated support 
$(\w{S}_k)_{k \in [K]}$. This construction is detailed in Section~\ref{subsec:ERMLasso}.
\end{itemize}


\subsection{Estimation of the active sets}
\label{subsec:estimationSupport}

Our classification procedure relies on the estimation of the active sets $S_k^*$ for all $k \in [K]$.
To this aim, we consider the least squares contrast with a Lasso penalty for repeated observations. The considered contrast is an adaptation of the penalized criteria proposed in~\cite{BGM} in the context of repeated observations with a fixed horizon time of observation $T$.

Let $k \in [K]$. Hereafter, we define the estimator of the active set $S_k^*$.
We denote $\left(\mathcal{T}_T^{(1)}, \ldots, \mathcal{T}_T^{(n_k)}\right)$ the observations from class $k$ coming from $\mathcal{D}_n^{(2)}$, 
with $n_k = \sum_{i=1}^n \one_{\{Y^{(i)}=k\}}$ the random number of observations from class $k$.
First, we define the considered contrast. To this end, we introduce the generic parameter
$\theta = \left(\theta_{1}, \ldots, \theta_{M}\right)^\prime \in \mathbb{R}^{M(M+1)}$,
such that for each $j \in [M]$, $\theta_j = (\theta_{j,\ell})_{0 \leq \ell \leq M}$ writes as
\begin{equation*}
\theta_j := \left(\mu_j, a_{j,1}, \ldots, a_{j,M}\right).    
\end{equation*}
The vector of true parameters is also denoted by $\theta^*_k = \left(\theta_{k,1}^*, \ldots, \theta^*_{k,M}\right)^\prime \in \mathbb{R}^{M(M+1)}$. For each $j \in [M]$, it expresses as follows 
\begin{equation}\label{eq:theta}
\theta_{k,j}^* = \left(\theta^*_{k,j,\ell}\right)_{l \in \left\{0,\ldots, M\right\}} := \left(\mu^*_{k,j}, a^*_{k,j,1}, \ldots, a^*_{k,j,M}\right)^\prime \in \mathbb{R}^{M+1}.
\end{equation}
Then, for each $\theta \in \mathbb{R}^{M(M+1)}$, and $i \in [n_k]$, we define the corresponding intensity function associated to the observation $\mathcal{T}_{T}^{(i)}$ that stems from class $k$ for $t \in[0,T]$ as
\begin{equation*}
\lambda^{(i)}_{k,j,\theta}(t)  = \mu_{k,j} + \sum_{j'=1}^M  a_{k,j,j'}  \sum_{T^{(i)}_{\ell} \in \mathcal{T}^{(i)}_{t,j'}} h(t-T^{(i)}_{ \ell}).
\end{equation*}

The considered penalized contrast is  defined, for each $\theta \in \R^{M(M+1)}$,
as follows,
\begin{equation}
\label{eq:contrastLasso}
R_{T,n_k}(\theta) :=  \frac{\one_{\left(n_k \geq 1\right)}}{n_k}\sum_{i=1}^{n_k} \left(\frac{1}{T}\sum_{j=1}^M \int_0^T \lambda^{(i)2}_{k,j,\theta}(t) dt-\frac{2}{T}\sum_{j=1}^M \sum_{T^{(i)}_{\ell} \in \mathcal{T}_{T,j}^{(i)}}\lambda^{(i)}_{k,j,\theta}\left(T^{(i)}_{\ell}\right)\right).  
\end{equation}
Note that if $n_k= 0$, we have $\w{\theta} = 0$. 
The Lasso estimator is then defined as
\begin{equation}\label{eq:thetahat}
\w{\theta}_k \in \argmin {\theta \in \R^{M(M+1)}} \left\{R_{T,n_k}(\theta)+ \kappa \sum_{j = 1}^M \sum_{j'=1}^M |\theta_{j,j'}|\right\}. 
\end{equation}
Finally, from the estimator $\w{\theta}_k$, we get the estimated support of $A_k^*$
\begin{equation*}
\w{S}_k = \{(j,j^{'}) \in  [M]^2, \;\; \w{\theta}_{k,j,j^{'}} \neq 0\}.
\end{equation*}
Note that $\w{S}_k$ represents the estimated active set of $A_k^*$ since it does not involve the first column of $\w{\theta}_k$ that contains the vector of estimated baseline $(\mu_j)_{j \in[M]}$.


\subsection{ERM classifier with refitting step}
\label{subsec:ERMLasso}

In this section, we present the last step of our estimation procedure, which is dedicated to the construction of the final classifier. In particular, it involves the estimation of parameter $\theta^* = \left(\theta_k^*\right)_{k \in  [K]}$.
We highlight that this step relies on the estimated support $\w{S}_k$.
To this end, we introduce the constraint set of parameters
\begin{equation*}
\Theta_n := \left\{\theta = (\mu, A) \in \mathbb{R}_{+}^{M}\times \mathbb{R}_{+}^{M^2}, \ \mu_j \in \left[\frac{1}{n}, \log(n)\right], \ j \in [M], \ \left\|A\right\|_F \leq \log(n) \right\},
\end{equation*}
and finally the set of interest
\begin{equation}\label{eq:thetaHatSet}
\w{\Theta} := \left\{\theta = \left(\theta_1, \ldots, \theta_K\right) \in \Theta_n^K, \;\; {\rm supp}(A_k)= \w{S}_k\right\}.    
\end{equation}
Several comments can be made from the definition of the set of parameters $\w{\Theta}$. First we observe that conditional on the event $\{\hat{S}_k = S_k^*\}$, for $n$ large enough, the true parameter $\theta^*$ belongs to the set $\w{\Theta}$.
Indeed, in view of Assumption~\ref{ass:mu}, for $n$ large enough,
we may assume that $1/n < \mu_0 < \mu_1 < \log(n)$, and $C_A  \leq n$. Furthermore, we emphasize that the choice of the bounds on the coefficients on the parameters of $\w{\Theta}$ allows  to get rid of the unknown constants defined in Assumption~\ref{ass:mu}.
These choices are also driven by technical aspects. In particular, such bounds are required to apply concentration arguments.  
Additionally, let us mention that contrary to the previous step, the optimization is performed on $\R_+$ for each coefficient.

Let us present the estimation of the parameter $\theta^*$ and then the construction of the resulting classifier $\w{g}$. This construction follows the strategy provided in~\cite{DDS} for $M=1$, and is based on the dataset $\mathcal{D}_n^{(2)}$.
It relies on the empirical risk minimization principle. 
Specifically, for each $\theta \in \widehat{\Theta}$, we introduce an associated score functions $f_{\theta} = (f^1_{\theta}, \ldots, f^K_{\theta})$ such that for an observed sequence of events $\mathcal{T}_T$
\begin{equation*}
f_{\theta} (\mathcal{T}_T) = 2\pi_{k,\w{p}, \theta}(\mathcal{T}_T) -1, ~k \in [K],
\end{equation*}
with 
$$
\pi_{k,\w{p}, \theta}(\mathcal{T}_T) 
 = \frac{ \w{p}_k{\rm e}^{F_{k}(\mathcal{T}_T)}}{\sum_{k'=1}^K \w{p}_{k'} {\rm e}^{F_{{k'}}(\mathcal{T}_T)}}
$$
and
\begin{equation*}
F_{k,\theta}(\mathcal{T}_T) = -\sum_{j=1}^M  \int_0^T \lambda_{k, j,\theta}(s) d s 
 + \sum_{j=1}^M \sum_{T_\ell\in \mathcal{T}_{T,j}} \log\left(\lambda_{k, j, \theta}(T_\ell )\right).
\end{equation*}
Note that the form of the score function $f_{\theta}$ is chosen according to the result provided in Proposition~\ref{prop:prop1}. 
Let $\theta \in \w{\Theta}$, and $f_{\theta}$ its associated score function, we define its empirical $L_2$-risk as
\begin{equation*}
\w{\cR}_2\left(f_{\theta}\right) := \dfrac{1}{n} \sum_{i=1}^n \sum_{k = 1}^K \left(Z_k^{(i)}-f_{\theta}^k(\cT_T^{(i)})\right)^2, \quad Z_k^{(i)} = 2\one_{\{Y_i=k\}}-1.
\end{equation*}
Then, we define the estimator of $\theta^*$ as the minimizer of the empirical $L_2$-risk, 
\begin{equation}\label{eq:thetaERM}
\w{\theta}^{\rm R} \in \argmin{\theta \in \w{\Theta}} \w{\cR}_2(f_{\theta}) .
\end{equation}
From the estimator $\w{\theta}^{\rm R}$ of parameter $\theta^*$, we define the \texttt{ERMLR} classifier as follows
\begin{equation}\label{eq:classifERM}
\w{g}(\mathcal{T}_T) \in  \argmax{k \in \mathcal{Y}} \pi_{k,\w{p}, \w{\theta}^{\rm R}}(\mathcal{T}_T) 
\end{equation}
Note that for computational purpose, as it is usual in classification, the $0-1$ loss is then replaced by the $L_2$ convex surrogate \citep[see \textit{e.g}][]{Zhang04}.  
In particular, the $L_2$-loss is classification calibrated and
Zhang's lemma~\cite{Zhang04} ensures that 
\begin{equation*}
{\E}\left[\cR(\w{g})-\cR(g^*)\right] \leq \frac{1}{\sqrt{2}} \big( \E \left[ \cR_2(f_{\w{\theta}^{\rm R}})- \cR_2(f_{\theta^*}) \right]\big)^{1/2},    
\end{equation*}  
with $\cR_2$ the oracle counterpart of the considered empirical risk $\w{\cR}_2$ defined as
\begin{equation*}
\cR_2(f_{\theta}) = \E\left[\left(Z_k- f_{\theta}(\cT_T)\right)^2\right], \;\; {\rm with}\;\; Z_k = 2\one_{\{Y=k\}}-1.   
\end{equation*}

One of the main appealing property of our classification algorithm is that we take advantage of the estimated support to perform the minimization of the empirical $L_2$-risk on a set of parameter whose dimension is much smaller than $M^2$. Besides, rather than using the estimated parameters obtained at the first step (Lasso-step), we consider the estimator of parameter $\theta^*$ as the minimizer of loss adapted to our multiclass classification setting.


%% file: sec4_theorems.tex

In this section, we first provide the consistency of the estimator of the active set in Section~\ref{subsec:LassoResult}. Then, in Section~\ref{subsec:convclassif}, we derive the rate of convergence of our classification procedure with respect to the misclassification risk. 


\subsection{Support recovery for classification}
\label{subsec:LassoResult}

In this section, we present the key result of the Lasso procedure. More precisely, we show that
\begin{equation*}
\mathbb{P}\left(\hat{S}_k = S_k^* \right) \rightarrow 1, \;\; n \rightarrow \infty,
\end{equation*}
which implies that for each $k \in [K]$, the Lasso estimator $\hat{\theta}_k$ solution of Equation~\eqref{eq:thetahat} has nonzero entries at the same positions as the true parameter 
$\theta_k^*$. In particular, for the multivariate Hawkes process, for $j,j' \in [M]^2$, the Lasso step can be interpreted as interaction selection, where the objective is to select whether a component $j$ is impacted by a component $j'$.

Before, to give our main result, 
we introduce some notations for the Lagrangian version of the Lasso criterion given in Equation \eqref{eq:thetahat}.


\paragraph{Notations.}

In the rest of the section, we fix a class $k \in [K]$, and for simplicity we drop the dependency on $k$. Besides, throughout this section, we work conditional on the event $n_k \geq 1$. We also remind the reader that $n_k$ is the random number of observations from class $k$ in the dataset $\mathcal{D}_n^{(1)}$ of size $n$.
Then, we define for each $t \in (0,T]$ the random matrix 
$\mathbb{H}_t \in \mathbb{R}^{n_k \times (M+1)}$ as follows
\begin{equation}\label{eq:H}
(\mathbb{H}_t)_{i,j}  = H_j^{(i)}(t),  \;\; {\rm with} \;\;  H^{(i)}_{j}(t):=\int_0^t h(t-s)\dd N^{(i)}_{j}(s),~ j \neq 0,~ H_0^{(i)} \equiv 1.
\end{equation}
From the definition of the matrix $\mathbb{H}_t$, we observe that 
\begin{equation*}
\lambda^{(i)}_{j,\theta}(t)= \sum_{j'=0}^M H^{(i)}_{j'}(t) \theta_{j,j'} ,
\end{equation*}
in other words,
\begin{equation*}
\mathbb{H}_t \theta_j  = \left(\lambda_{j,\theta}^{(i)}(t)\right)_{i =1, \ldots,n_k}.
\end{equation*}
For $j \in [M]$, and $i \in \{1,\ldots,n_k\}$, we consider
$M^{(i)}_j$ the martingale associated to the counting process $N^{(i)}_j$ through the Doob-Meyer decomposition. 
We then denote $\dd M(t) = \left(\dd M_j^{(i)}(t)\right)_{j,i} \in \mathbb{R}^{M \times n_k}$, and define the random martingale matrix $Z$  as
\begin{equation*}
Z:= \int_{0}^T \left(\dd M(t) \mathbb{H}_t\right)^\prime. 
\end{equation*}
Besides, the $j$-th column of $Z$ is denoted by $Z_j$. 
Therefore, for $j,j^{'} \in [M] \times \{0, \ldots,M\}$, the main term of $Z$ is the continuous-time martingale,
\begin{equation}
\label{eq:Zjj}
Z_{j,j^{'}} = \sum_{i=1}^{n_k} \int_{0}^T
H_{j^{'}}^{(i)}(t) \dd M_{j}^{(i)}(t).
\end{equation} 
We finally define the random matrix $\mathbb{H}$ of size $(M+1) \times (M+1)$ as
$$\mathbb{H} := \dfrac{1}{T} \int_{0}^T \mathbb{H}_t^\prime\mathbb{H}_t \dd t.$$

In the following, for $S \subset [M]$, we denote $\mathbb{H}_{S,S}$ the matrix where the lines and the columns are restricted to the set $S$. 


\paragraph*{Assumptions.}

Classical conditions in the $\ell_1$ constraint framework are considered as, for instance, in~\cite{VDG} and references therein.
According to Equation~\eqref{eq:theta}, the true parameter is denoted $\theta_{j}^* = \left(\theta^*_{j,\ell}\right)_{\ell \in \left\{0,\ldots, M\right\}}$, and 
$\theta_{j}^*=\left(\mu_{j}^*, a^*_{j,1}, \ldots, a^*_{j,M}\right)^\prime \in \mathbb{R}^{M+1}$. For each $j \in [M]$, we also denote $S_{{\theta_j}}^{*}$  
the active set of $\theta_j^*$. Note that, since $\mu_j^* > 0$, it contains at least one element.

The first assumption is the mutual incoherence, which is also referred as irrepresentability condition. Heuristically, this imposes that the correlation between the non-active and active variables must not be higher than the variations within the actives variables, otherwise the lasso estimator would not be able to dissociate them. It involves an incoherence parameter $\gamma \in (0,1]$ that must not be too small.

\begin{ass}[Mutual incoherence (MI)]\label{ass:MI}
There exists some $0 < \gamma \leq 1$ such that, a.s.
\begin{equation*}
\max_{j=1, \ldots, M}\|\mathbb{H}_{S_{\theta_j}^{*c}, S_{\theta_j}^*}  \mathbb{H}_{S_{\theta_j}^{*}, S_{\theta_j}^*}^{-1}\|_{\infty} \leq 1- \gamma.
\end{equation*}
\end{ass}

The following condition ensures that the submatrix $\mathbb{H}_{S_{\theta_j}^{*}, S_{\theta_j}^*}$ does not have its columns linearly dependent (in which case it could be impossible to estimate $\theta^*$ when the true active set is known). The notation $\Lambda_{\min}$ denotes the minimal eigenvalue.

\begin{ass}[Minimum eigenvalue (ME)]\label{ass:ME}
There exists $\Lambda_0 > 0$ such that, a.s.,
\begin{equation*}
\min_{j=1, \ldots, M}\Lambda_{\min}\left(\dfrac{\mathbb{H}_{S_{{\theta_j}}^{*}, S_{\theta_j}^*}}{n_k}\right) \geq \Lambda_0.
\end{equation*}
\end{ass}

Finally, the last condition of minimum signal ensures that the non-zero entries of the true coefficients are large enough to be properly estimated. Specifically, it imposes that the minimum value
of the true parameter restricted to the support $S^*$
cannot decay to zero faster than the regularization parameter, 
$\kappa$ which is specified in Theorem~\ref{prop:convsupport}.

\begin{ass}[Minimum signal condition (MS)]\label{ass:MS}
\begin{equation*}
\min_{j,j^{'} \in S^*} \left|\theta^*_{j,j^{'}}\right| > \Lambda_0 \max_{j=1, \ldots, M}
\sqrt{\left|S_{\theta_j}^*\right|} \dfrac{\log^4(nM^2)}{\sqrt{n}}.
\end{equation*}
\end{ass}


\paragraph*{Support recovery result.}

The result provided by Theorem~\ref{prop:convsupport} is the main ingredient to derive rate of convergence of our classification procedure. Nevertheless, it is an interesting result {\it per se}. 
Under the above assumptions, for each class $k \in [K]$, we establish the uniqueness of the Lasso solution, the consistency of the estimated support, and  the uniform consistency of the estimator of $\theta^*$.

\begin{theo}\label{prop:convsupport}
Assume that $n > \frac{2}{p_0}$, and let $\kappa  =  \dfrac{\log^4(nM^2)}{\sqrt{n}}$. Grant Assumptions~(MI), (ME), and (MS). 
There exists an event $\Omega_{n}$ with $\P(\Omega_{n}) \geq 1-\dfrac{C}{n}$, on which $n_k \geq 1$, and
\begin{equation*}
\min_{\theta \in \R^{M(M+1)}}  \left\{R_{T,n_k}(\theta) + \kappa \sum_{j = 1}^M \sum_{j'=1}^M |\theta_{j,j'}|  \right\},
\end{equation*}
where $R_{T, n_k}$ is given in Equation \eqref{eq:contrastLasso}, admits a unique solution $\w{\theta}$ which satisfies the following
\begin{enumerate}[label=(\roman*)]
\item ${\rm supp}(\w{\theta}) = {\rm supp}(\theta^*)$;
\item \[
\left\|\w{\theta} - \theta^*\right\|_{\infty}  \leq \dfrac{\Lambda_0\max_{j=1, \ldots, M}\sqrt{|S_{\theta_j}^*|}\log^4(nM^2)}{\sqrt{n}}.
\]
\end{enumerate}
\end{theo}

Several comments can be made from the above result.
First,  a straightforward consequence of Theorem~\ref{prop:convsupport}, is that for each $k \in [K]$,
\begin{equation*}
\mathbb{P}\left(\hat{S}_k = S_k^*\right) \geq 1-\dfrac{C}{n}.
\end{equation*}
Hence, our result provides rate of convergence for the estimator of the support $\hat{S}_k$.
Furthermore, in view of Assumption~\ref{ass:MS}, we have that on the event $\Omega_n$, $\hat{\theta}_{j,j'} > 0$. 
Notably, Theorem~\ref{prop:convsupport} extends the result of~\cite{BGM} in the context of repeated observations with fix observation time. In particular, the work of~\cite{BGM} does not provide support recovery result. However, we emphasize that our result requires stronger assumption than in~\cite{BGM}. 
Let us notice that the result holds also for $\kappa $ larger that $ {\log^4(nM^2)}/{\sqrt{n}}$ but in this case the rates of convergence is slower.

Second, up to logarithmic factor, the condition on the tuning parameter $\kappa= \kappa_n$ is of the same order as in \cite{wainwright2009sharp}. Besides, up to a logarithmic factor,  we obtain a rate of convergence of order $\max_{j=1, \ldots, M}|\sqrt{S_{\theta_j}^*|}/\sqrt{n}$ in sup-norm for the estimator $\hat{\theta}$, we can note that this rate is of the same order than the one that would expect in the classical Gaussian framework~\cite{VDG}. 
We also highlight that in the logarithmic factor, the power of the log term is in part due to the fact that the number of jump-times of the process is not bounded a.s.

Finally, the proof of this result is based on a preliminary lemma, which gives a control in probability of the maximum of the martingale terms $Z_{j,j'}$ defined in Equation~\eqref{eq:Zjj}.
This inequality is obtained using a Bernstein type inequality proven in~\cite{BGM}. This data-driven inequality and the sub-exponential property of the counting process (see Assumption~\ref{ass:expomoments}) lead to the concentration result. Then, we follow the primal-dual-witness method of proof \citep[see for instance][]{tibshirani2015sparsity}. 


\subsection{Rate of convergence of the \texttt{ERMLR} classifier}\label{subsec:convclassif}

In this section, we derive theoretical property of the \texttt{ERMLR}
algorithm $\hat{g}$. To establish our result, we take advantage of the support recovery result provided in Section~\ref{subsec:LassoResult}. On the set $\{\hat{S}_k = S_k^*\}$, the excess risk of $\hat{g}$ is upper-bounded by applying classical arguments derived from the classification framework. While we use Theorem~\ref{prop:convsupport} to bound the excess risk on the event
$\{\hat{S}_k \neq S_k^*\}$.
Then, we obtain the following result.

\begin{theo}\label{thm:riskERMsupport}
Grant Assumptions~\ref{ass:prob}, \ref{ass:h} and \ref{ass:mu}. For $n$ large enough,
there exists a constant $C>0$ such that,
\begin{equation*}
\E\left[\cR(\w{g}) - \cR(g^*)\right] \leq C \left(\frac{ (M+s^*) \log(nM)}{n}\right)^{1/2},
\end{equation*}
where $C$ depends on $T, K, \|h\|_\infty, \mu_0, \mu_1, p_0$.
\end{theo}

As expected, we 
highlight that, thanks to the Lasso step, we manage to obtain, up to a logarithmic factor, a rate of order $\sqrt{(M+s^*)/n}$ rather than $\sqrt{M+M^2)/n}$. Notably, we show that the proposed algorithm achieves the usual parametric rate.


%% file: sec5_implementation.tex

In this section, a comprehensive description of the implementation details is specified. As the \texttt{ERMLR} procedure execution involves two minimization problems, these two steps are described separately in Section~\ref{subsec:implemLasso} and Section~\ref{subsec:implemERM}. In both cases, each choice is discussed in terms of the state of the art and its relevance in the context of its use. Besides, let us highlight that the implementation of the procedure relies on state-of-the art algorithms and \texttt{C++} codes wrapped in \texttt{Python} which serves the purpose of rapid computation. 

\subsection{Implementation details for the Lasso step}
\label{subsec:implemLasso}

For the support recovery step, our strategy consists in the minimization of the least squares contrast with Lasso penalty defined in Equation~\eqref{eq:thetahat}. This objective function is written as the sum of two functions. While the least squares contrast is differentiable, convex and smooth (\textit{i.e.}\,with a Lipschitz continuous gradient), the $\ell_1$-norm is non-differentiable at zero. To this extent, to carry out the minimization of such objective function, we use first-order optimization algorithm based on proximal methods with Nesterov's momentum method, namely \texttt{FISTA}, see \cite{beck2009}. Compared to the classical proximal algorithm, the construction of a new iterate of the descent is based on a specific linear combination of the previous two points. This makes \texttt{FISTA} benefits from a significantly faster rate of convergence.
A recommended choice of the descent step is $1/L$ with $L$ the Lipschitz constant of the gradient.
We stop the descent after $200$ iterations if the stopping criterion, based on relative distance between two successive iterations, is not fulfilled yet.

Another important aspect of the Lasso step concerns the calibration of the penalization constant $\kappa$ which controls the regularization.
As our goal is to recover the true support, $\kappa$ must be large enough to set all non-active coefficients to zero. To this end, our strategy is the following: different values of $\kappa$ are explored through a grid of sufficiently fine size, denoted $\Delta$, and the one that minimizes a specific model selection criterion is chosen. The criterion used here is the Extended Bayesian Information Criteria (\texttt{EBIC}) introduced by \cite{chen2008}. For some $\gamma \in [0, 1]$ and $\kappa \in \Delta$, this criterion takes the following form:
$$\texttt{EBIC}_{\gamma}(\kappa) := -2 L_{T, n}\left(\w{\theta}(\kappa)\right) + \left\lvert S_{\w{\theta}(\kappa)} \right\rvert \log(n) + 2\gamma\log\left(\binom{M^2}{\left\lvert S_{\w{\theta}(\kappa)} \right\rvert}\right)$$
where $\w{\theta}({\kappa})$ is the Lasso estimated with the tuning parameter $\kappa$, $L_{T, n}$ is the $\log$-likelihood of the model, $\left\lvert S_{\w{\theta}(\kappa)} \right\rvert$ is the size of this support, namely the number of active coefficients of $\w{\theta}(\kappa)$.

Compared to a classical \texttt{BIC} criteria (namely $\gamma=0$), an additional penalization is added to 
take into account the number of possible active sets of the same size. As this quantity is also increasing with this size, it seems to be very relevant in a high-dimensional setting. In the following, we 
choose $\gamma=1$ and $\lvert \Delta \rvert = 40$ as exploration grid size. 

Finally, let us highlight that for both the least squares contrast and the $\log$-likelihood functional, computation such as gradient or loss evaluation are optimized and implemented in \texttt{C++} which serves the purpose of rapid computation. 


\subsection{Implementation details for the ERM step}
\label{subsec:implemERM}

For the classification step, our strategy consists in minimizing the convexified empirical risk defined in Equation~\eqref{eq:classifERM}. According to the definition of the constraint set of parameters defined in Equation~\eqref{eq:thetaHatSet}, each coefficient must be positive. To ensure that each coefficient $a_{k,j,j'}$ remains in $[0,c]$, with $c>0$, the minimization is done under inequality constraints and we use a projected gradient descent algorithm. Nevertheless, since this objective function is non-smooth and non-convex {\it w.r.t.} to the coefficients, its minimization requires particular care. 
In particular, the tuning of the step-size in the descent is very tricky \footnote{Furthermore, classical method such as backtracking line-search with Armijo-Wolfe condition cannot be used due to the piece-wise constant nature of the projection operator (see \cite{ferry2023}).}.
On the other hand, adaptive gradient methods, such as \texttt{AdaGrad} \citep[see][]{duchi2011}, have been widely used in large-scale optimization due to their ability to adjust the step size for each feature according to the geometry of the problem. In practice, \texttt{AdaGrad} is known to be an efficient method in non-convex setting (in particular for training deep neural networks optimization, see \cite{gupta2014}). In addition, some theoretical guarantees for the convergence of \texttt{AdaGrad} for non-convex functions have been provided in the literature \citep[see][]{ward2020,wang2023}.
With this in mind, we use a parameter-free projected adaptive gradient descent method, in the inspiration of \texttt{AdaGrad}, called \texttt{Free AdaGrad} and introduced in \cite{chzhen2023}. 
Compared with the classical algorithm, its main advantage lies in the fact that it is adaptive to the distance between the initialization and the optimum, and to the sum of the square norm of the gradients.
The initial starting point is chosen as the estimate given by the Lasso step, the initial guess for the distance between the starting point and the optimum is taken as $\gamma_0=0.1$ and we stop the descent after $1000$ iterations if the stopping criteria described before is not fulfilled yet. 

%% file: sec6_num.tex
 
The goal of this section is to investigate the performance of our method from a numerical standpoint using synthetic data. First, in Section~\ref{subsec:benchmark}, alternative strategies are proposed for comparison purpose. Then, the simulation and evaluation scheme is thoroughly detailed in Section~\ref{subsec:simulScheme} and in Section~\ref{subsec:evalScheme}. Finally, the obtained results, for support recovery by the Lasso step in Section~\ref{subsec:resultsLasso} and the classification procedure performance in Section~\ref{subsec:resultsClassif} are presented. 

\subsection{Benchmark}
\label{subsec:benchmark}

Let us detail here the different competitors which are compared with our classifier. 

\paragraph{Simple plug-in strategy.}

A full plug-in strategy consists in use the estimators $\w{\theta}$ of the parameters, obtained by minimizing the least-squares contrast with Lasso penalty on the adjacency matrix given in Equation~\eqref{eq:thetahat}. Then, we plug $\hat{\theta}$ into the Bayes classifier formula. Consequently, the resulting classifier for a new observation $\cT_T$ is
$$\w{g}_{\w{p}, \w{\theta}}(\cT_T) = \argmax{k \in \cY} \frac{\w{p}_k e^{F_{\w{\mu}_k, \w{A}_k}(\cT_T)}}{
\sum_{k'=1}^K \w{p}_{k'} e^{F_{\w{\mu}_{k'}, \w{A}_{k'}}(\cT_T)}
},
$$
where $\w{p}$ is the estimated distribution of $Y$. This classifier is learned on the entire training sample $\mathcal{D}_n$ of size $2n$. This classifier is referred as \texttt{PI}.

\paragraph{Oracle on estimated support.}

We are also interested in another predictor, referred to as \texttt{OES} for \textit{The Oracle on Estimated Support}, which is defined as follows 
$$\w{g}_{\w{p}, \theta^*_{\w{S}}}(\cT_T) \in \argmax{k \in \mathcal{Y}} \pi_{k,\w{p}, \theta^*_{\w{S}}}(\mathcal{T}_T) .$$
where 
\begin{equation*}
(\theta^*_{\w{S}})_{k,j,j'} := \begin{cases}
\theta^*_{k,j,j'} & \text{if} \ (j,j') \in \w{S}_k \\
0 & \text{otherwise}
\end{cases}.
\end{equation*}
It corresponds to the best possible predictor that relies on the support recovered in the Lasso step. Note that if the true support is recovered by the Lasso step, then it exactly corresponds to the Bayes rule. By taking into account this predictor, we can quantify the effect of poor support recovery in terms of classification error, while evaluating the gain that could be obtained by an \texttt{ERM} step. 


\subsection{Simulation scheme}
\label{subsec:simulScheme}

In this section, we give some details on the panel of scenarios on which our Lasso estimator and our classifier are evaluated. 

\paragraph{MHP path generation.}

Concerning synthetic data generation, each path is simulated using cluster representation algorithm (see \cite{moller2005}). This sampling procedure relies on the branching structure of the MHP, that can be viewed as Poisson cluster process. We consider the classical choice of exponential kernel $h(s) = \beta\exp(-\beta s)$ with $\beta=3$.

\paragraph{Scenarios.}

We consider two scenarios, referred to as \textit{Scenario 1} and \textit{Scenario 2}. In both scenarios, different structures of the interaction matrix $A^*$ are explored. In \textit{Scenario 1}, $A^*$ is chosen to be a diagonal block matrix. In addition to self-exciting interaction, the block structure models interaction between a group of connected components. Coefficient values, which gives the intensity of influence, are the same within each block, but vary from one to another. For a larger value of $M$, the blocks are expanded so that the parsimony rate remains the same for each value of $M$.
In \textit{Scenario 2}, the coefficients are chosen randomly with different values. Due to the randomness of the choice of the active set, the diagonal coefficients may be all set to zero. Thus, there may be no self-excitation
in this case.
In both scenarios, the vector of exogenous intensity $\mu^*$ is chosen as constant for each component, meaning that spontaneous events occur in the same way for each individual. In Figure~\ref{fig:exemplesA}, a visual representation of theses scenarios, for $M=25$ is given in the form of a heat map. In particular, the values of the coefficients of the matrix $A^*$ are given by the color bar. We precise also the sparsity rate, which is the $\%$ of zero-coefficients in the matrix, i.e. $|S^{*c}|/M^2$. 

\begin{figure}[hbtp]
    \centering
    \begin{subfigure}{0.45\textwidth}
        \centering
        \resizebox{0.9\textwidth}{!}{\includegraphics[scale=0.6]{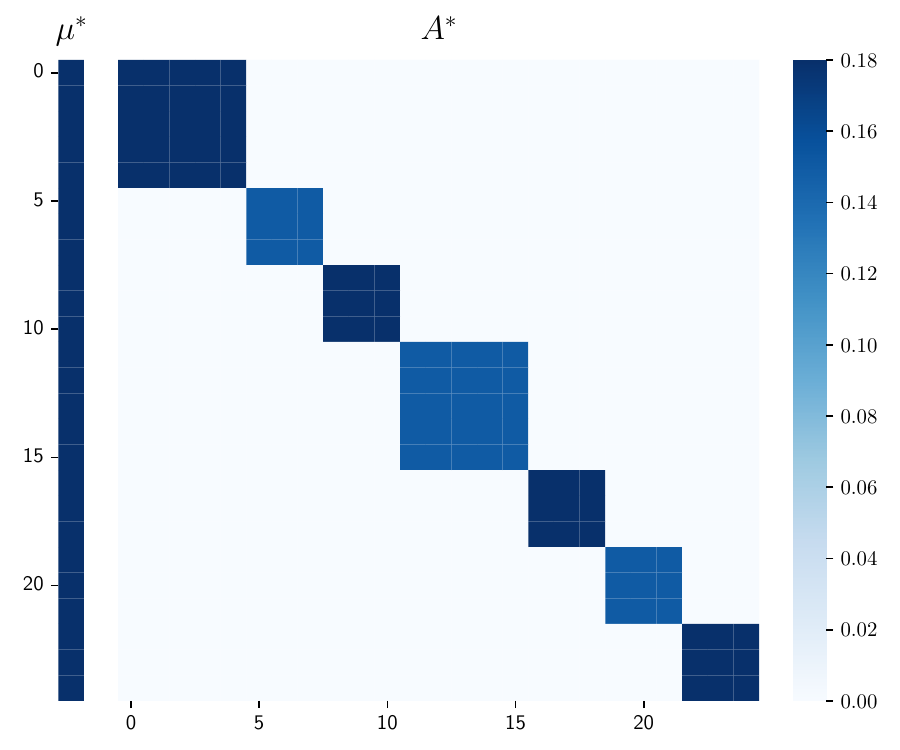}}
        \caption{\textit{Scenario 1}}
    \end{subfigure}
    \hfill
    \begin{subfigure}{0.45\textwidth}
        \centering
        \resizebox{0.9\textwidth}{!}{\includegraphics[scale=0.6]{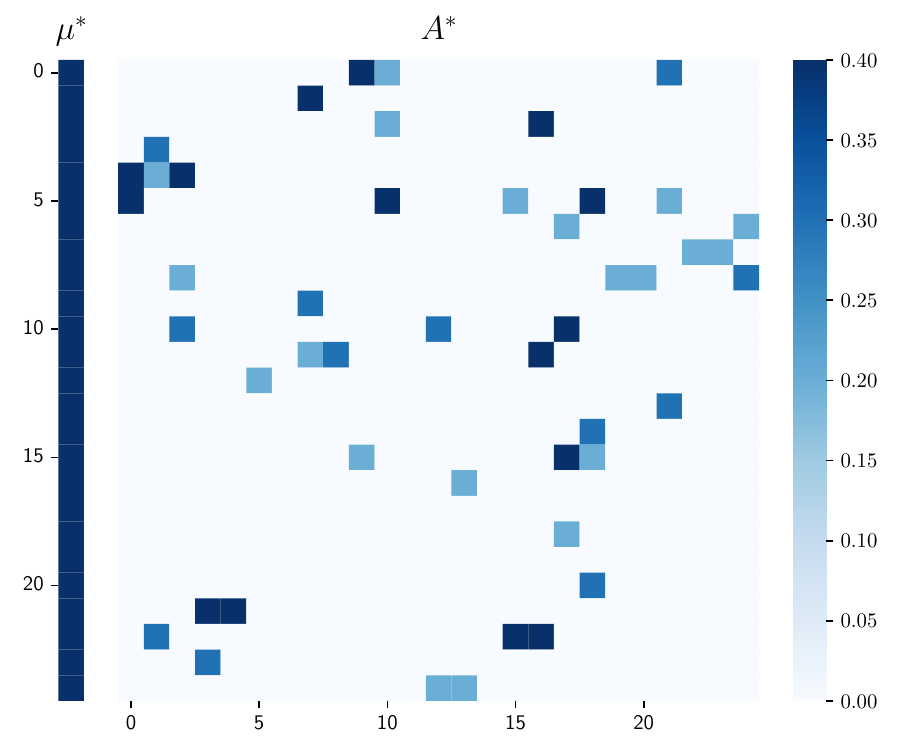}}
        \caption{\textit{Scenario 2}}
    \end{subfigure}
\caption{Visualization of $\theta^* = (\mu^*, A^*)$ in both scenarios for $M=25$. The exogenous intensity for each component: 0.4. Sparsity rate of $A^*$ in \textit{Scenario 1}: 85\%; in \textit{Scenario 2}: 92\%.}
\label{fig:exemplesA}
\end{figure}

To illustrate the classification task, we consider the $3$-classes classification setting, i.e. $K=3$. The three classes are created on the basis of the two scenarios described above. For \textit{Scenario 1}, the blocks of different size are interchanged, as well as the values of the coefficients within them. For \textit{Scenario 2}, based on the same support for each class, the values of the coefficient are interchanged. In both cases, the resulting classes are quite balanced and close from each other. Finally, the exogenous intensity is chosen to be the same for each of the three classes.
In Table~\ref{tab:scenariosclassif}, we give for each scenario, the values of the Frobenius norm, the spectral radius and the sparsity rate as a function of the dimension and of the label. For terminology convenience, we refer to the classification scenario resulting from \textit{Scenario 1} (resp. \textit{Scenario 2}) as \textit{Scenario 1} (resp. \textit{Scenario 2}). 

\begin{table}[hbtp]
    \centering
    \resizebox{0.9\textwidth}{!}{
    \begin{tabular}{|c|c|ccc|ccc|}
    \hline
 \multirow{2}{*}{} & \multirow{2}{*}{}
    & \multicolumn{3}{c |}{Scenario 1} 
    & \multicolumn{3}{c |}{Scenario 2} \\
  &  & $k=1$ & $k=2$ & $k=3$ & $k=1$ & $k=2$ & $k=3$\\
    \hline
  \multirow{3}{*}{M=10} &  $\|A_k^*\|_F$ & 1.37 & 1.37 & 1.39 & 1.44 & 1.54 & 1.31\\
   & $\rho(A_k^*)$ & 0.76 & 0.76 & 0.76 & 0.00 & 0.00 & 0.00 \\
  &  $|S_k^{*c}|$ & 0.86 & 0.86 & 0.86 & 0.89 & 0.89 & 0.89 \\    \hline
\multirow{3}{*}{M=25}   & $\|A_k^*\|_F$ & 1.63 & 1.63 & 1.68 & 2.07 & 2.25 & 2.07 \\
 &   $\rho(A_k^*)$ & 0.90 & 0.90 & 0.90 & 0.50 & 0.55 & 0.44 \\
 &   $|S_k^{*c}|$ & 0.85 & 0.85 & 0.85 & 0.92 & 0.92 & 0.92\\
    \hline 
\multirow{3}{*}{$M=50$}&    $\|A_k^*\|_F$ & 1.52 & 1.52 & 1.52 & 2.55 & 2.77 & 2.42 \\
  &  $\rho(A_k^*)$ & 0.90 & 0.90 & 0.90 & 0.68 & 0.74 & 0.63 \\
 &   $|S_k^{*c}|$ & 0.85 & 0.85 & 0.85 & 0.94 & 0.94 & 0.94\\
    \hline
    \end{tabular}}
    \caption{Presentation of the different scenarios with $K=3$. For each class $k$, the Frobenius norm, the spectral radius and the sparsity rate of $A^*_k$ are specified.}
    \label{tab:scenariosclassif}
\end{table}


\subsection{Evaluation scheme}
\label{subsec:evalScheme}

Hereafter, we present the evaluation scheme that relies on Monte-Carlo repetitions. We fix $T=5$, and ${p}^*\sim \mathcal{U}_{[3]}$. For each scenario described, each value of $M\in\{10, 25, 50\}$, and each value of $n\in\{300, 600, 1500\}$, we 
repeat independently $30$ times the following steps.
\begin{enumerate}
    \item Simulate the data set $\cD_{n_{\textrm{train}}}$ and $\cD_{n_{\textrm{test}}}$;
    \item Based on $\cD_{n_{\textrm{train}}}$, for each $k=1, \dots, K$ compute $\w{p}_k=\frac{1}{n}\sum_{i=1}^{n}\one_{\left\{Y^i=k\right\}}$;
    \item Based on $\cD_{n_{\textrm{train}}}$, Lasso step:
    \begin{enumerate}
    \item For each $k \in [K]$, calibrate the penalization constant using $\texttt{EBIC}_1$ criteria by exploring values in the grid $\Delta$. For each $\kappa \in \Delta$ do:
    \begin{enumerate}
        \item using \texttt{FISTA}, compute $\w{\theta}_k$ the Lasso estimate with tuning parameter $\kappa$;
        \item based on $\w{\theta}_k$, compute $\texttt{EBIC}_{1}(\kappa)$;
    \end{enumerate}
    and choose $\w{\kappa}_k \in \argmin {\kappa \in \Delta} \texttt{EBIC}_{1}(\kappa)$;
    \item Given $(\w{\kappa}_k)_{k \in \mathcal{Y}}$, for each $k=1, \dots, K$ do:
    \begin{enumerate}
        \item using \texttt{FISTA}, compute the Lasso estimates $\w{\theta}_k$ with tuning parameter $\w{\kappa}_k$;
        \item get the estimated support $\w{S}_k = \left\{(j,j^{'}) \in [M], \;\; \w{\theta}_{k,j,j^{'}} \neq 0\right\}.$
    \end{enumerate}
    \item From $(\w{S}_k)_{k \in \mathcal{Y}}$ compute the classifier $\w{g}_{\texttt{OES}}$, from $(\w{\theta}_k)_{k \in \mathcal{Y}}$ compute the classifier $\w{g}_{\texttt{PI}}$
    \end{enumerate}
    \item For one arbitrary class $k \in \mathcal{Y}$, assess the quality of the support recovery using Hamming distance and $\ell_2$ distance defined as
    \begin{equation*}
        d_H\left(A_k^*, \w{A}_k\right) = \frac{1}{M^2}\sum_{j,j'=1}^M \mathds{1}_{\left\{A_{k,j,j'}^* \neq \widehat{A}_{k,j,j'}\right\}}, \;\textrm{ and }\;
        d_{\ell_2}\left(A_k^*, \w{A}_k\right) = \sqrt{\sum_{j,j'=1}^M {\lvert A_{k,j,j'}^* - \widehat{A}_{k,j,j'} \rvert}^2};
    \end{equation*}
    \item From $\cD_{n_{\textrm{train}}}$, perform the \texttt{ERM} step:
    \begin{enumerate}
        \item starting from $(\w{\theta}_k)_{k \in \mathcal{Y}}$ as the initial point, we minimize the $L_2$-risk defined in Equation \eqref{eq:thetaERM} using \texttt{Free AdaGrad} to obtain $(\w{\theta}^{\rm R}_k)_{k \in \mathcal{Y}}$;
        \item from $\w{\theta}^{\rm R}$ and $\w{p}$ we build the classifiers $\w{g}_{\texttt{ERMLR}}$.
    \end{enumerate}
    \item Based on $\cD_{n_{\textrm{test}}} = \{(\mathcal{T}_T^{(i)}, Y^i), i = 1, \ldots, n_{{\rm test}}\}$,
 evaluate the error rate of the classifiers $\texttt{PI}$  and $\texttt{ERMLR}$ using
    \begin{equation*}
    {\rm Err}_{\texttt{PI}} = \dfrac{1}{n_{\rm test}} \sum_{i =1}^{n_{\rm test}} \mathds{1}_{\{\w{g}_{\texttt{PI}}(\cT_T^i) \neq Y^i\}}, \;\textrm{ and }\; {\rm Err}_{\texttt{ERMLR}} = \dfrac{1}{n_{\rm test}} \sum_{i =1}^{n_{\rm test}} \mathds{1}_{\{\w{g}_{\texttt{ERMLR}}(\cT_T^i) \neq Y^i\}};
    \end{equation*}
    \end{enumerate}


\subsection{Numerical Results for support recovery}
\label{subsec:resultsLasso}

This section is devoted to the discussion of the obtained results of the Lasso procedure. These results are provided in Table~\ref{tab:evalLasso}, in Table~\ref{tab:evalLasso2}, in Figure~\ref{fig:visualrecoverysupport} and in Figure~\ref{fig:Lassoexecutiontime}.

\begin{table}[hbtp]
    \centering
    \resizebox{\columnwidth}{!}{%
    \begin{tabular}{ | c | c | c | c | c | c | c | c | c | c | c | c |}
        \hline
        & \multirow{2}{*}{M}& \multicolumn{3}{c |}{$d_H$}    
        & \multicolumn{3}{c |}{$d_{\ell_2}$}\\
        \cline{3-8}
        & & $n=100$ & $n=500$ & $n=1000$ & $n=100$ & $n=500$ & $n=1000$\\
        \hline
        \multirow{3}{*}{\textit{Scenario 1}} 
        & 10 & 0.04 (0.03) & 0.02 (0.02) & 0.02 (0.02) & 0.39 (0.07) & 0.18 (0.04) & 0.13 (0.02) \\ 
        & 25 & 0.04 (0.01) & 0.03 (0.01) & 0.03 (0.01) & 0.91 (0.07) & 0.40 (0.04) & 0.29 (0.02) \\ 
        & 50 & 0.11 (0.01) & 0.07 (0.00) & 0.07 (0.00) & 1.80 (0.12) & 1.60 (0.02) & 1.64 (0.02) \\ 
        \hline
        \multirow{3}{*}{\textit{Scenario 2}} 
        & 10 & 0.04 (0.02) & 0.03 (0.02) & 0.03 (0.02) & 0.43 (0.07) & 0.20 (0.03) & 0.14 (0.02) \\ 
        & 25 & 0.03 (0.01) & 0.03 (0.01) & 0.03 (0.01) & 0.96 (0.11) & 0.44 (0.04) & 0.32 (0.04) \\ 
        & 50 & 0.03 (0.00) & 0.02 (0.00) & 0.01 (0.00) & 1.76 (0.09) & 0.94 (0.07) & 0.68 (0.04)\\ 
        \hline
    \end{tabular}%
    }
    \caption{Lasso results over 30 Monte-Carlo repetitions for both scenarios for three value of $M$. The impact of $n$ is investigated. The standard deviation is provided between parentheses. $T=5$}
    \label{tab:evalLasso}
\end{table}

\begin{figure}[hbtp]
    \centering
     \resizebox{1.0\textwidth}{!}{\includegraphics[scale=0.6]{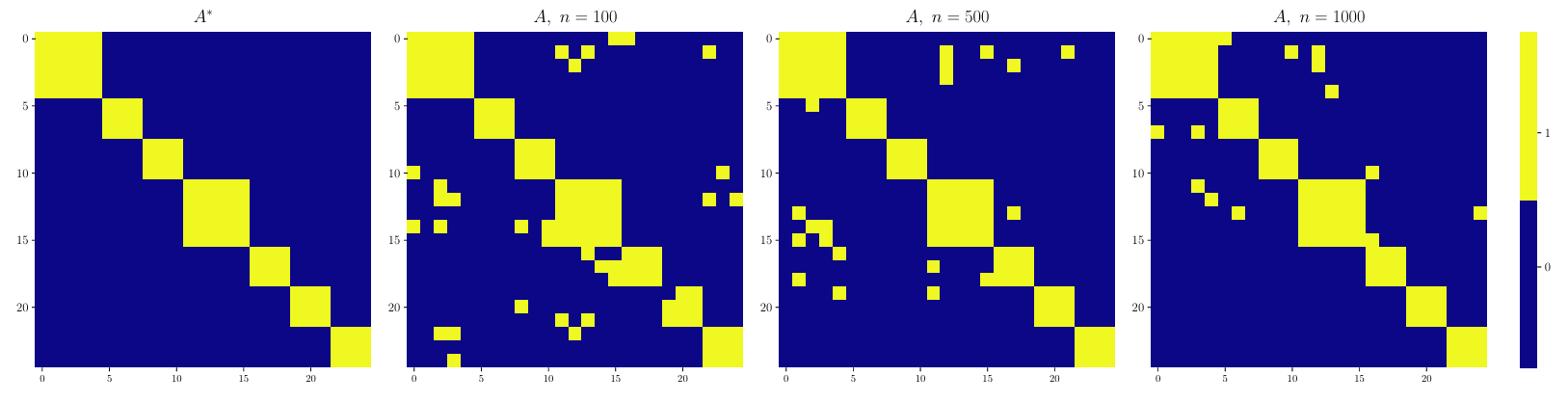}}
\caption{True support ${\rm supp}(\theta^*)$ and recovered support ${\rm supp}(\hat{\theta})$ in \textit{Scenario 1}. }
\label{fig:visualrecoverysupport}
\end{figure}

First, from Table~\ref{tab:evalLasso}, we can see that Hamming distance between the true support and the estimated one is close to zero in all settings, therefore our procedure is able to correctly recover the active set of $A^*$. This remains true even for small values of $n$ and for high-dimensional networks, for which the Hamming distance is quite small. 
As expected, the larger $n$ is, the better the support is reconstructed, whether in terms of Hamming distance or $\ell_2$ distance. Thus, in addition to reconstructing the support more accurately, a gain is also made in terms of point parameter estimation, illustrating the theoretical result of support consistency and convergence of the associated estimator established in Section~\ref{sec:theo}. In particular, for large value of $M$, such as $M=50$, a clear decrease in the Hamming distance is noticeable for increasing values of $n$. 
Finally, it is worth emphasizing that, in the case of \textit{Scenario 1}, the Lasso procedure is successful in recovering the underlying block structure of the interaction matrix $A^*$. This assertion is supported by the Figure~\ref{fig:visualrecoverysupport}, which visually shows the convergence of the support to the actual structure as the number of observations increases.  

\begin{figure}[hbtp]
    \centering
    \resizebox{0.5\textwidth}{!}{\includegraphics[scale=0.6]{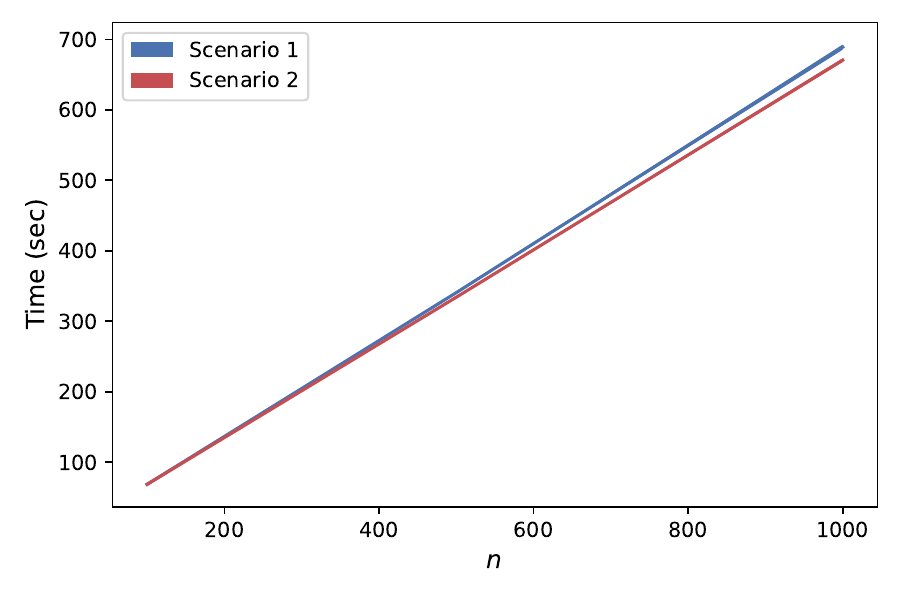}}
    \caption{Average execution time over $30$ repetitions of the entire Lasso procedure as a function of $n$ for \textit{Scenario 1} with $M=25$. The standard deviation is shown in shaded fill on either side of the curve.}
    \label{fig:Lassoexecutiontime}
\end{figure}

\begin{table}[hbtp]
    \centering
    \resizebox{\columnwidth}{!}{%
    \begin{tabular}{ | c | c | c | c | c | c | c | c | c | c | c | c |}
        \hline
        & \multirow{2}{*}{M}& \multicolumn{3}{c |}{\# events}    
        & \multicolumn{3}{c |}{time (sec)}\\
        \cline{3-8}
        & & $n=100$ & $n=500$ & $n=1000$ & $n=100$ & $n=500$ & $n=1000$\\
        \hline
        \multirow{3}{*}{\textit{Scenario 1}} 
        & 10 & 6330 (197) & 32079 (398) & 63905 (666) & 8.37 (0.09) & 42.12 (0.35) & 82.42 (0.53) \\ 
        & 25 & 14221 (397) & 71226 (773) & 142073 (1789) & 68.67 (0.46) & 340.62 (1.36) & 688.99 (3.70) \\ 
        & 50 & 12993 (513) & 64743 (839) & 129969 (1252) & 447.22 (4.21) & 2290.80 (22.63) & 4519.32 (49.73) \\ 
        \hline
        \multirow{3}{*}{\textit{Scenario 2}} 
        & 10 & 4692 (125) & 23367 (250) & 46570 (402) & 8.17 (0.09) & 39.81 (0.34) & 80.35 (0.58) \\ 
        & 25 & 9524 (147) & 47651 (354) & 95737 (632) & 68.64 (0.22) & 334.13 (1.30) & 670.51 (2.67) \\ 
        & 50 & 12363 (254) & 62228 (575) & 124604 (1142) & 459.69 (1.36) & 2268.71 (7.69) & 4522.17 (16.79)\\ 
        \hline
    \end{tabular}%
    }
    \caption{Number of observed events, average execution time over 30 Monte-Carlo repetitions for both scenarios. The standard deviation is provided between parentheses. $T=5$}
    \label{tab:evalLasso2}
\end{table} 

Now let us discuss the computational cost of our procedure. It can be seen from Figure~\ref{fig:Lassoexecutiontime} and Table~\ref{tab:evalLasso2} that the execution time of the entire procedure is quite reasonable, even for a large value of $M$. It is important noting that the execution time also includes the choice of $\kappa$ with the \texttt{EBIC} criterion, and this with a grid of fine size $\lvert \Delta \rvert = 40$. Thus, we can afford to explore with great precision and still have a relatively short execution time. 
For comparison purposes, it is worth noting that, as we are dealing with short-time path repetitions, our observations would be equivalent to a unique path of horizon time of $n \times T$. Finally, as our procedure benefits from fast computational properties, it appears therefore realistic to apply it to large-scale networks. This could be a matter of great interest for real-world applications, which often involve a network of huge dimension. 


\subsection{Numerical results for classification}
\label{subsec:resultsClassif}
This section is devoted to the discussion of the obtained results of the \texttt{ERMLR} procedure. These results are provided in Table~\ref{tab:evalERMLR}, and Figure~\ref{fig:ERMLRErrorrate}.

\begin{table}[hbtp]
    \begin{subtable}[c]{1.0\textwidth}
    \centering
    \resizebox{0.9\columnwidth}{!}{%
    \begin{tabular}{ | c | c | c | c || c | c |}
        \hline
        & M & Bayes & \texttt{OES} & \texttt{PI} & \texttt{ERMLR} \\
        \hline
        \multirow{3}{*}{\textit{Scenario 1}} 
        & 10 & 0.134 (0.005) & 0.135 (0.006) & 0.155 (0.007) & \textbf{0.152} (0.007) \\ 
        & 25 & 0.087 (0.004) & 0.107 (0.013) & 0.143 (0.011) & \textbf{0.134} (0.011) \\
        & 50 & 0.092 (0.005) & 0.313 (0.05) & \textbf{0.218} (0.020) & 0.219 (0.019) \\
        \hline
        \multirow{3}{*}{\textit{Scenario 2}} 
        & 10 & 0.251 (0.007) & 0.255 (0.008) & \textbf{0.276} (0.012) & 0.278 (0.010) \\ 
        & 25 & 0.237 (0.008) & 0.260 (0.014) & \textbf{0.316} (0.016) & 0.326 (0.012) \\
        & 50 & 0.246 (0.008) & 0.406 (0.031) & \textbf{0.391} (0.026) & 0.410 (0.032) \\
        \hline
    \end{tabular}%
    }
    \caption{$n=300$}
    \end{subtable}
    \\
    \begin{subtable}[c]{1.0\textwidth}
    \centering
    \resizebox{0.9\columnwidth}{!}{%
    \begin{tabular}{ | c | c | c | c || c | c |}
        \hline
        & M & Bayes & \texttt{OES} & \texttt{PI} & \texttt{ERMLR} \\
        \hline
        \multirow{3}{*}{\textit{Scenario 1}} 
        & 10 & 0.135 (0.006) & 0.135 (0.006) & 0.146 (0.007) & \textbf{0.144} (0.008) \\ 
        & 25 & 0.086 (0.004) & 0.087 (0.004) & 0.118 (0.005) & \textbf{0.113} (0.005) \\
        & 50  & 0.091 (0.004) & 0.189 (0.021) & 0.183 (0.008) & \textbf{0.179} (0.009) \\
        \hline
        \multirow{3}{*}{\textit{Scenario 2}} 
        & 10 & 0.247 (0.008) & 0.248 (0.008) & \textbf{0.260} (0.009) & 0.262 (0.008) \\ 
        & 25 & 0.236 (0.007) & 0.237 (0.008) & 0.276 (0.010) & \textbf{0.276 }(0.010) \\
        & 50 & 0.245 (0.008) & 0.309 (0.014) & \textbf{0.333} (0.016) & 0.349 (0.020) \\
        \hline
    \end{tabular}%
    }
    \caption{$n=600$}
    \end{subtable}
    \\
    \begin{subtable}[c]{1.0\textwidth}
    \centering
    \resizebox{0.9\columnwidth}{!}{%
    \begin{tabular}{ | c | c | c | c || c | c |}
        \hline
        & M & Bayes & \texttt{OES} & \texttt{PI} & \texttt{ERMLR} \\
        \hline
        \multirow{3}{*}{\textit{Scenario 1}} 
        & 10 & 0.135 (0.005) & 0.136 (0.005) & 0.139 (0.005) & \textbf{0.139 (0.006)} \\ 
        & 25 & 0.087 (0.006) & 0.087 (0.005) & 0.100 (0.006) & \textbf{0.098} (0.006) \\
        & 50 & 0.093 (0.005) & 0.183 (0.08) & 0.179 (0.07) & \textbf{0.173} (0.08) \\
        \hline
        \multirow{3}{*}{\textit{Scenario 2}} 
        & 10 & 0.253 (0.009) & 0.253 (0.009) & \textbf{0.257} (0.009) & 0.259 (0.010) \\ 
        & 25 & 0.236 (0.009) & 0.236 (0.009) & \textbf{0.253} (0.008) & 0.254 (0.008) \\
        & 50 & 0.247 (0.008) & 0.251 (0.009) & \textbf{0.293} (0.012) & 0.296 (0.011) \\
        \hline
    \end{tabular}%
    }
    \caption{$n=1500$}
    \end{subtable}
    \caption{Empirical error over 30 Monte-Carlo repetitions for each classifier in the three scenarios for three values of $M$. The impact of $n$ is investigated. The standard deviation is provided between parentheses. The value of $n_{\rm test}=3000$ is chosen. $T=5$}
\label{tab:evalERMLR}
\end{table}

\begin{figure}[hbtp]
    \centering
    \resizebox{0.5\textwidth}{!}{\includegraphics[scale=0.6]{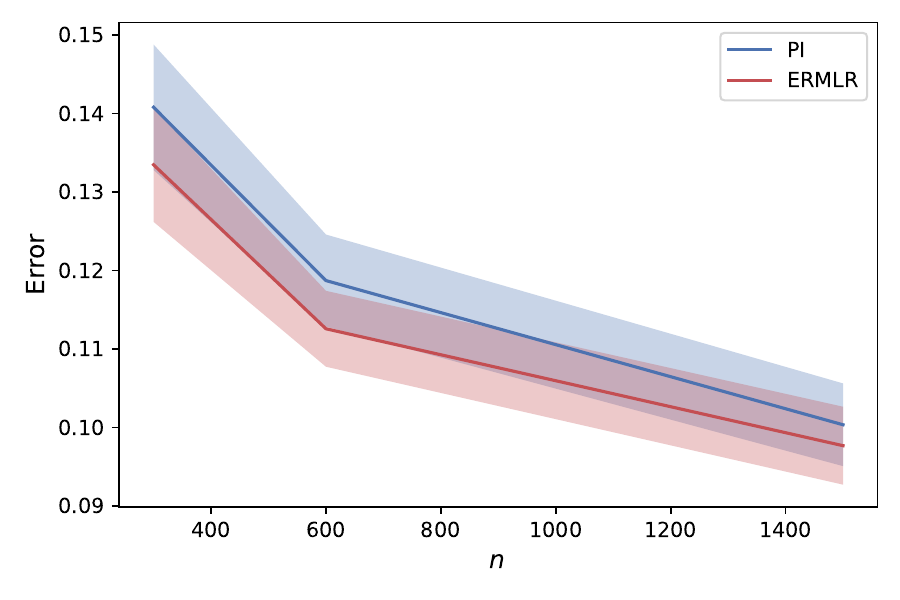}}
    \caption{Averaged Error rate over $30$ repetitions of the $\texttt{ERMLR}$ and $\texttt{PI}$ procedure as a function of $n$ for \textit{Scenario 1} with $M=25$. The standard deviation is shown in shaded fill on either side of the curve.}
    \label{fig:ERMLRErrorrate}
\end{figure}

First, from Table~\ref{tab:evalERMLR}, we can see that the \texttt{ERMLR} is close to the Bayes classifier in terms of error rate, in both scenarios and for each value of $M$. In particular, note that for $n=1500$, its error rate is almost equal to that of the Bayes classifier. In fact, as expected the greater the number of data, the closer the classifier comes to the Bayes classifier, which illustrate the consistency of the \texttt{ERMLR} procedure established in Section~\ref{sec:theo}. This decreasing tendency of the error rate of \texttt{ERMLR} is illustrated in \textit{Scenario 1} with $M=25$ in Figure~\ref{fig:ERMLRErrorrate}. 

Another important point is the comparison with the \texttt{PI} classifier as a benchmark. Overall, it can be seen that the \texttt{PI} exhibits good performance. This can be explained by the fact that recovering the true support structure is sufficient for accurate class prediction. On the other hand, poor support recovery also impacts the performance of the \texttt{ERMLR} predictor. This gap can be quantified with the \texttt{OES} oracle classifier,  which gives the gain that could be obtained by an
\texttt{ERM} step. For these reasons, it is not expected to see a big gap between the two. Nevertheless, it is worth noting that, in case of \textit{Scenario 1}, a significant gain by the \texttt{ERM} refitting step can be observed. This assertion is supported by Figure~\ref{fig:ERMLRErrorrate}, where it can be seen that the \texttt{ERMLR} classifier is better in terms of error rate. This suggests that, for some particular structures, a refitting step leading to a finer point estimate of the parameters is relevant and leads to better performance. 


%% file: sec7_discussion.tex

In the present work, we propose a novel classification algorithm tailored to classify Multivariate Hawkes Processes paths in high-dimension. 
For each class, a first step is dedicated to the sparse estimation of  the support of the adjacency matrix. Then, in a second step, we build a classifier that takes of advantage of the estimated support. Specifically, the resulting classifier is based on the minimization of a \texttt{ERM} criterion.
We establish rates of convergence for both estimated support and 
classification algorithm. Finally, we illustrate the numerical performance of our procedure through a comprehensive simulation study.

A possible guideline for further research is to consider a more challenging model by including inhibition interaction. From a theoretical aspect, it may be tricky since adding inhibition effect induces complication due to the non linearity of the underlying intensity  function. In particular, providing a closed form of the compensator is a key aspect to compute the least-square contrast or the likelihood function. The work of~\cite{bonnet2022} and~\cite{bonnet2023inference}
should form a theoretical basis for this future work.
From a practical point of view, 
a procedure which is able to deal with inhibition, may be applied to generalize the work of~\cite{DDBLSB}. 
Indeed, the use of MHP allows to model simultaneously different species echolocation calls and then the effects of inter-species cooperation. Furthermore,  adding inhibition effects, potentially translates the ecological aspect of inter-species competition.

Another direction could be to investigate a penalized \texttt{ERM} classifier. It would allow to deal with the high-dimensional setting without the prior Lasso step. Indeed,
this procedure relies on a global penalized criterion dedicated to the classification task. This direction is left for further investigations. 

Finally, from a practical standpoint, \texttt{sparkle}, a full Python library for Hawkes process inference in high-dimension and classification is in development. It consists in a toolkit for Hawkes process modeling which relies on \texttt{C++} codes wrapped in \texttt{Python} for fast computation. 

%% file: sec8_appendix.tex

\textit{This appendix gathers the proofs of the theoretical results of the paper.
It is organized as follows.
Appendix~\ref{app:techRes} provides useful technical results.
The proof of the closed-form expression of the Bayes classifier is established in Appendix~\ref{app:Bayes}.
The proof of the support recovery result is given in Appendix~\ref{app:supportRec}.
Finally, the rate of convergence of the \texttt{ERMLR} algorithm is proved in
Appendix~\ref{app:RatesofCve}.}

\textit{Throughout the proofs, the notation $C$ refers to a generic positive constant, which may differ from line to line. In particular, this generic constant $C$ does not depend on $n$ or on the dimension $M$. However, it may depend on the other parameters. For the sake of simplicity we denote $\cT$ for $\cT_T$.}

\section{Technical results}
\label{app:techRes}

\setcounter{lemme}{0}
\renewcommand{\thelemme}{A.\arabic{lemme}}
\setcounter{prop}{0}
\renewcommand{\theprop}{A.\arabic{prop}}

\begin{prop}
\label{prop:excessrisk}
For any classifier $g \in \mathcal{G}$, we have
\begin{equation*}
\cR(g)-\cR(g^*) =\mathbb{E}\left[\sum_{1 \leq i \neq k \leq K} |\pi^*_i(\cT)-\pi^*_k(\cT)|\one_{\{g^*(\cT)=i, g(\cT)=k\}}\right].    
\end{equation*}
\end{prop}

\begin{proof}
This result is established by \cite{DDS}.
Let $g \in \mathcal{G}$ a classifier. We observe that
\begin{equation*}
\mathcal{R}(g) = \mathbb{E}\left[\one_{\{g(\cT) \neq Y\}}\right] = 1- \mathbb{E}\left[\one_{\{g(\cT) = Y\}}\right]  =1-\mathbb{E}\left[\pi^*_{g(\cT)}\right].  
\end{equation*}
Therefore, from the above equation and the definition of the Bayes classifier $g^*$, we get
\begin{equation*}
\cR(g)-\cR(g^*) = \mathbb{E}\left[\left|\pi^*_{g^*(\cT)}- \pi^*_{g(\cT)}\right|\right].
\end{equation*}
Since for each $g \in \mathcal{G}$, $g(\cT) = \sum_{k=1}^K k \one_{\{g(\cT) =k\}}$,
the above equation yields the result.
\end{proof}

\begin{lemme}
\label{lem:boundInftyMatrix}
Let $A \in \mathbb{R}^{d \times d}$ (symmetric), and $X \in \mathbb{R}^d$.
Then,
\begin{equation*}
\left\|AX\right\|_{\infty} \leq \sqrt{d} \rho(A) \left\|X\right\|_{\infty}    
\end{equation*}
\end{lemme}


\begin{lemme}[Hoeffding]
\label{lem:Hoeffding}
Let $B\sim \mathcal{B}(n,p)$, with $p \in (0,1)$. We then have for all $t >0$ and $n > \frac{t}{p}$,  
\begin{equation*}
\mathbb{P}(B \leq t) \leq \exp{\left( -2n(p-t/n)^2\right)} .
\end{equation*}
\end{lemme}


\section{Proof for Bayes classifier}
\label{app:Bayes}

\setcounter{lemme}{0}
\renewcommand{\thelemme}{B.\arabic{lemme}}
\setcounter{prop}{0}
\renewcommand{\theprop}{B.\arabic{prop}}

We first denote for all $k\in\cY$
$$\Phi^k_t:=\frac{\dd{\mathbb P}_k|_{{\mathcal F}_t^N}}{\dd{\mathbb P}_0|_{{\mathcal F}^N_t}},$$
with $\mathcal{F}_T^N:=\sigma\left(\cT_T\right) = \sigma\left(N_t, 0 \leq t \leq T\right)$.
We classically obtain:
$$\log(\Phi^k_t) = -\sum_{j=1}^M \int_0^t (\lambda^{*}_{k,j}(s)-1) \dd s +  \int_0^t\log( \lambda^{*}_{k,j}(s) ) \dd N_{j}(s),$$
by writing {\it w.r.t.} a Poisson process measure of intensity 1  \citep[see Chapter 13 of][]{DVJ}.
Thus, for $t\geq 0$, we have the following equation for the mixture measure 
$$
d{\mathbb P}|_{{\mathcal F}^N_t} = \sum_{k=1}^K p_k \dd{\mathbb P}_k|_{{\mathcal F}^N_t} = \sum_{k=1}^K p_k \Phi^k_t \dd{\mathbb P}_0|_{{\mathcal F}^N_t}
$$
and then
$$
\frac{\dd{\mathbb P}_k|_{{\mathcal F}^N_t}}{\dd{\mathbb P}|_{{\mathcal{F}}^N_t}} = \frac{p_k\Phi_t^k \dd{\mathbb P}_0|_{{\mathcal F}^N_t}}{\sum_{j=1}^K p_j \Phi^j_t \dd{\mathbb P}_0|_{{\mathcal F}^N_t}} = \frac{p_k\Phi_t^k}{\sum_{j=1}^K p_j \Phi^j_t}. 
$$
Finally, by using \eqref{def:F}, it comes 
$\pi^*_k\left(\mathcal{T}_T\right) = \frac{p^*_k{\rm e}^{F_k^*}}{\sum_{j=1}^K p^*_j {\rm e}^{F_j^*}}$, that concludes the proof.

\section{Proofs for support recovery}
\label{app:supportRec}

\setcounter{lemme}{0}
\renewcommand{\thelemme}{C.\arabic{lemme}}
\setcounter{prop}{0}
\renewcommand{\theprop}{C.\arabic{prop}}

In this section, we gather the proof of the result provided in Section~\ref{subsec:LassoResult}. We first recall and introduce the main notations for the proof of the main result in Section~\ref{subsec:appNotations}.
Then, in section~\ref{subsec:BernsteinApp} we establish a Bernstein lemma. This lemma is the cornerstone of the proof of the support recovery which is given in Section~\ref{subsec:mainApp}.

\subsection{Notations} 
\label{subsec:appNotations}

We recall that the learning sample is $D_{n}= \{\left(\mathcal{T}_T^i,Y_i\right), \ldots, \left(\mathcal{T}_{T}^{n}, Y_n\right)\}$.
Let $k \in \mathcal{Y}$ be a fixed integer. Throughout this section,  all the results are established for a generic class $k$.
Let us define the random variables 
\begin{equation*}
n_k = \sum_{i = 1}^n \one_{Y^{(i)}=k}.
\end{equation*}
Hence $n_k \sim \mathcal{B}(n,p_k^*)$. 
We also recall that $\min_{k \in [K]} p_k^* \geq p_0 > 0$.\\

For sake of simplicity, we remove the dependency {\it w.r.t.} $k$.
To sum up, our parameters of interests are $\mu, A$, and we at our disposal a sample of (random) size $n_k$.
In the rest of this section, we work {\bf conditional on} $\{n_k \geq 1\}$.

\subsection{A Bernstein lemma}
\label{subsec:BernsteinApp}


\begin{lemme}[Bernstein Lemma]\label{lem:Z}
Assume that $n \geq \dfrac{2}{p_0^*}$.
Let us define the event 
\begin{equation*}
\Omega_n  := \left\{\dfrac{1}{n_k}\max_{j,j^{'}}\left|Z_{j,j{'}}\right| \leq  C \dfrac{\log^{3}\left(nM^2\right)}{\sqrt{n}} \right\} \bigcap \left\{n_k \geq \dfrac{np_k^*}{2}\right\}. 
\end{equation*}
There exists $C_{\left\|h\right\|_{\infty}, p_0^*} > 0$, such that $\mathbb{P}\left(\Omega_n\right) \geq 1- \dfrac{C_{\left\|h\right\|_{\infty}, p_0^*}}{n}$.
\end{lemme}

\begin{proof}
Fore clarity of presentation, the proof is divided in two steps. 


\paragraph*{First step.}
In this step we work on the event $\{n_k \geq 1\}$ and conditional on $\one_{\{Y_1 = k\}}, \ldots, \one_{\{Y_n = k\}}$.
For $j,j^{'} \in [M] \times \left(\{0\} \cup [M]\right)$, we apply Theorem~4 in~\cite{BGM} to the real valued random variable $Z_{j,j^{'}}$. 
For clarity, we consider the same notations as in~\cite{BGM}. 

To this end,  for a fixed $(j,j^{'}) \in [M] \times \left(\{0\} \cup [M]\right)$
and $t \in [0;T]$,
we define the tensor (see~\cite{BGM} for its definition and related properties) $\mathbb{T}_t$ of shape $1 \times 1 \times M \times n_k$ as follows 
\begin{equation}\label{eq:tensor}
\mathbb{(T}_{t})_{1,1,k,\ell} = 
\begin{cases}
H_{j^{'}}^{(\ell)}(t) \;\; {\rm if} \;\; k = j\\
0 \;\; {\rm else},
\end{cases}
\end{equation}
for  $k \in [M]$ and $\ell \in [n_k]$.
We also recall that the matrix $\dd M(t)$ is defined by the main term $\dd M(t)_{j,i} = \dd M_{j}^{(i)}(t)$.
According to~\cite{BGM} we have that $Z_{j,j^{'}} = Z_{\mathbb{T}}(T) \in \mathbb{R}$  defined by
\begin{equation*}
Z_{\mathbb{T}_t}(T) = \int_{0}^T \mathbb{T}_{t} \circ dM_t
\end{equation*}    
satisfies 
\begin{equation}\label{eq:newZT}
Z_{\mathbb{T}_t}(T) = 
\sum_{k=1}^{M}\sum_{i = 1}^{n_k} \int_{0}^T
\mathbb{(T}_{t})_{1,1,k,i} dM_{k,i}(t)
=  \sum_{i=1}^{n_k} \int_{0}^T
H_{j^{'}}^{(i)}(t) \dd M_{j}^{(i)}(t).
\end{equation}
Furthermore, we observe that since the tensor $\mathbb{T}_t$ is symmetric we have
\begin{equation*}
\widehat{V}_{\mathbb{T}}(t) :=  \int_0^t \mathbb{T}_s^2 \circ \dd N_s = \sum_{i = 1}^{n_k} \int_0^t \left(H_{j^{'}}^{(i)}(s)\right)^2 \mathrm{d}N_j^{(i)}(s),
\end{equation*}
and
\begin{equation*}
b_{\mathbb{T}_t}:= \sup_{0\leq s \leq t} \max( \| \mathbb{T}_s\|_{\text{op}, \infty} \| \mathbb{T}_s^\prime\|_{\text{op}, \infty})
=  \sup_{0\leq s \leq t}\max_{i = 1, \ldots,n_k} \left|H_{j^{'}}^{(i)}(s)\right|
\end{equation*}
which both depend on $(j,j')$. 

Applying Theorem 4 of~\cite{BGM} on the event $\{n_k \geq 1\}$ and conditional on $\one_{\{Y_1 = k\}}, \ldots, \one_{\{Y_n = k\}}$, we then obtain that for $x > 0$ with probability at least $1-C\exp(-x)$ the following holds
\begin{equation}\label{eq:eqBernsteinOp1}
\left|Z_{j,j^{'}}\right|
\leq 2\sqrt{\lambda_{\max} (\w{V}_{\mathbb{T}_T}) (x+\ell_x(T))}
+ c(x+\ell_x(T))\left(1+b_{\mathbb{T}_T}\right), 
\end{equation}
since for all $(j,j^{'}) \in [M] \times \{0, \ldots M\}$,
\begin{equation}
\label{eq:VT}
\lambda_{\max}\left(\widehat{V}_{\mathbb{T}}(T)\right) \leq \widehat{V}_{\infty} :=\max_{i, j^{'}} \left(H_{j^{'}}^{(i)}(T)\right)^2\max_{i,j} N_j^{(i)}(T), 
\end{equation}
and
\begin{equation}
\label{eq:bt}
b_{\mathbb{T}_T} \leq b_{\infty}:=  \max_{i,j^{'}} \left|H_{j^{'}}^{(i)}(T)\right|,  
\end{equation}
from Equation~\eqref{eq:eqBernsteinOp1}, setting $x = \log(nM^2)$, with an union bound on $j,j^{'}$ we
obtain that the event
\begin{equation*}
\mathcal{Z} = \left\{\max_{j,j^{'}} \left|Z_{j,j^{'}}\right| \leq  2\sqrt{ \w{V}_{\infty} \left(\log(nM^2)+\ell_{\infty}\right)} + c\left(\log(nM^2)+\ell_{\infty}\right)\left(1+b_{\infty}\right)\right\},
\end{equation*}
with 
\begin{equation}
\label{eq:lt}
\ell_{\infty} =   2\log\log\left(\dfrac{4\w{V}_{\infty}}{\log(nM^2)} \vee 2\right) + 2\log\log\left(4b_{\infty }\vee 2\right),  
\end{equation}
satisfies
\begin{equation*}
\one_{\left\{n_k \geq 1 \right\}} \mathbb{P}\left(\mathcal{Z}^{c} | \one_{\{Y^{(1)} = k\}}, \ldots, \one_{\{Y^{(n)} = k\}}\right)  \leq  \one_{\left\{n_k \geq 1 \right\}} \dfrac{C}{n} \leq \dfrac{C}{n}. 
\end{equation*}
From the above inequality, we deduce that
\begin{eqnarray}
\label{eq:eqZeventC}
\nonumber
\mathbb{P}\left(\mathcal{Z}^{c}\right)  & = &\mathbb{P}\left(\mathcal{Z}^{c}, n_k \geq 1 \right) +
 \mathbb{P}\left(\mathcal{Z}^{c}, n_k =0  \right)\\ \nonumber
 &\leq & \dfrac{C}{n}+ \mathbb{P}\left(n_k = 0\right)\\ \nonumber
  & \leq &    \dfrac{C}{n} + \exp(n\log(1-p_k)) \\ 
 &\leq &     \dfrac{C}{n} + \exp(n\log(1-p_0) \leq \dfrac{C}{n}.
\end{eqnarray}

\paragraph{Second step.}

In this step, we provide a bound for $\w{V}_{\infty}$, $b_{\infty}$, and $\ell_{\infty}$ respectively defined in Equation~\eqref{eq:VT}, ~\eqref{eq:bt}, and~\eqref{eq:lt}.
To this end, we introduce the event 
\begin{equation*}
\Omega = \left\{n_k \geq \dfrac{np_k}{2}\right\} \bigcap \left\{\one_{\{n_k \geq 1\}}\max_{i,j} N_j^{(i)}(T) \leq \log^{5/3}(Mn)\right\}.
\end{equation*}
Note that, in view of the definition of $H_{j'}^{(i)}$, we have that on the event $\{n_k \geq 1\}$, we have
\begin{equation*}
\w{V}_{\infty} \leq \max\left(n_k \left\|h\right\|_{\infty} \max_{i,j} \left(N_j^{(i)}(T)\right)^3, n_k \max_{i,j} (N_j^{(i)}(T) \right) \leq C_{\left\|h\right\|_{\infty}} n_k\log^{5}(n).
\end{equation*}
With the same idea, we have that $b_{\infty} \leq  C_{\left\|h\right\|_{\infty}} \log^{5/3}(n)$.
Finally, we observe that $\ell_{\infty} \leq 2\log(nM^2)$ (as $M \geq 2$). Hence,
on the event $ \Omega \cap \mathcal{Z}$, it holds that $n_k \geq 1$ (since $n \geq \frac{2}{p_0})$, and,
\begin{equation*}
\dfrac{1}{n_k} \max_{j,j^{'}} \left|Z_{j,j^{'}} \right| \leq C \dfrac{\log^{3}\left(nM^2\right)}{\sqrt{n_k}}\leq C \dfrac{\log^{3}\left(nM^2\right)}{\sqrt{np_k} }\leq C \dfrac{\log^{3}\left(nM^2\right)}{\sqrt{n p_0}}.
\end{equation*}
To conclude the proof, since $ \mathbb{P}\left(\Omega^{c}_n\right) \leq \mathbb{P}\left(\left(\mathcal{Z}\cap \Omega\right)^{c}\right)$, it remains to control
$\mathbb{P}\left(\left(\mathcal{Z}\cap \Omega\right)^{c}\right)$.

Conditional on $\one_{\{Y_1 = k\}}, \ldots, \one_{\{Y_n = k\}}$, on the event $ \{n_k \geq 1\}$, applying the sub-exponential property of $N_j^{(i)}$, and Proposition~2.7.1 in~\cite{vershynin_2018}, we get
\begin{eqnarray*}
\mathbb{P}\left(\max_{i,j} N_j^{(i)}(T) > \log^{5/3}(Mn)\right) &\leq & Mn_k\exp\left(-c\log^{5/3}(nM)\right)\\
& \leq& \dfrac{1}{n}.    
\end{eqnarray*}
Therefore, from Lemma~\ref{lem:Hoeffding},
\begin{eqnarray*}
\mathbb{P}\left(\Omega^{c}\right) & \leq & \dfrac{1}{n} +\mathbb{P}\left(n_k \leq \dfrac{np_k}{2} \right)  \leq \dfrac{1}{n} + \exp\left(-n\frac{p_{0}}{2}\right) \\
&\leq &\dfrac{1}{n}.
\end{eqnarray*}
Finally, combining the last equation with Equation~\eqref{eq:eqZeventC}, we deduce that,
\begin{equation*}
\mathbb{P}\left(\left(\mathcal{Z}\cap \Omega\right)^{c}\right) \leq \dfrac{C}{n},    
\end{equation*}
which yields the result.

\end{proof}

\subsection{Proof of the main result~\ref{prop:convsupport}}
\label{subsec:mainApp}


Throughout the proof, we work on the event 
\begin{equation*}
\Omega_n  := \left\{ \dfrac{1}{n_k}\max_{j,j^{'}}\left|Z_{j,j{'}}\right| \leq  C \dfrac{\log^{3}\left(nM^2\right)}{\sqrt{n}} \right\} \bigcap \left\{n_k \geq \dfrac{np_k}{2}\right\}. 
\end{equation*}
Note that on the event $\Omega_n$, since $n \geq \frac{2}{p_0}$, the random variable $n_k$ satisfies $n_k \geq 1$.


The proof follows the {\it primal-dual witness method} as in~\cite{tibshiraniWainwright} Chapter 11, and goes in several steps.
Let us consider the penalized contrast
\begin{equation}\label{eq:penalisedcontrast}
\mathcal{C}(\theta) := R_{T,n_k}(\theta) + \kappa \sum_{j = 1}^M \sum_{j'=1}^M |\theta_{j,j'}|.
\end{equation}
An element $z$ of the subgradient of $\mathcal{C}$ at some point $\theta$ writes
as follows 
\begin{equation*}
\nabla R_{T,n_k}(\theta) + \kappa z,    
\end{equation*}
where the concatenated vector $z$ is $z=(z_1, \ldots, z_M)^\prime$ with ${z}_{j,0} = 0$ and ${z}_{j,j'} = {\rm sign}(\theta_{j,j'})$ for $j' \geq 2$ (with the convention that ${\rm sign}(0) \in [-1,1]$).
We say that a pair $(\w{\theta}, \w{z})$
is optimal if it satisfies the following zero-subgradient equation
\begin{equation}
\label{eq:eqSubGradientOptimal}
\nabla R_{T,n_k}(\w{\theta}) + \kappa \w{z} = 0.
\end{equation}

\paragraph{First step.}

We first build an \enquote{oracle} pair $(\w{\theta}, \w{z})$
that satisfies Equation~\eqref{eq:eqSubGradientOptimal} and such that $\w{\theta}_{S^{*c}} = 0$. 
First we define $\w{\theta}$, and $\w{z}_{S^*}$ as follows.
\begin{enumerate}
\item $\w{\theta}_{S^{*c}} = 0$,
\item $\w{\theta}_{S^*} \in \argmin{\theta_{S^*}} \widetilde{R}_{T,n_k}(\theta_{S^*})+  \kappa \sum_{j = 1}^M \sum_{j' \in S_{\theta_j}^*} |\theta_{j,j'}|$,
where 
\begin{multline*}
\widetilde{R}_{T,n_k}(\theta_{S^*}) = 
 \frac{1}{n_kT}\sum_{i=1}^{n_k} \sum_{j=1}^M\int_{0}^T 
 \left(\sum_{j' \in S_{\theta_j}^*}{\theta}_{j,j'}H^{(i)}_{j'}(t)\right)^2 \dd t 
 - 2 \int_0^T\left(\sum_{j' \in S_{\theta_j}^*} \theta_{j,j'} H^{(i)}_{j'}(t)\right) \dd N^{(i)}_j(t) .
\end{multline*}
\end{enumerate}
In view of the above conditions, since $\w{\theta}_{S^*}$ is a minimizer, we have for each $j \in [M]$
\begin{equation*}
\left(\nabla R_{T,n_k}(\w{\theta})\right)_{S_{\theta_j}^*} + \kappa \w{z}_{S_{\theta_j}^*}  =  0.  
\end{equation*}
We then have to build for each $j \in [M]$,
$\w{z}_{S_j*c}$ such that
\begin{equation*}
\left(\nabla R_{T,n_k}(\w{\theta})\right)_{S_{\theta_j}^{*c}} +  \kappa\w{z}_{S_j*c}  =  0.   
\end{equation*}
Hence, from the above equations and from the notation given in Equation \eqref{eq:H}, we deduce 
that $(\w{\theta}, \w{z})$ must satisfies
\begin{equation*}
\dfrac{2}{n_k} \mathbb{H}_{S_{\theta_j}^{*c}, S_{\theta_j}^*}\left(\w{\theta}_j-\theta^*_j\right)_{S_j^{*}}  -\dfrac{2}{n_k} (Z_j)_{S_{\theta_j}^{*c}} + \kappa z_{S_{\theta_j}^{*c}}    = 0,
\end{equation*}
and
\begin{equation*}
\dfrac{2}{n_k} \mathbb{H}_{S_j^{*}, S_{\theta_j}^*}\left(\w{\theta}_j-\theta^*_j\right)_{S_j^{*}}  -\dfrac{2}{n_k} (Z_j)_{S_j^{*}} + \kappa z_{S_j^{*}}    = 0.
\end{equation*}
From the last equation, and as 
$z_{S_{\theta_j}^*}={\rm sign}((\w{\theta}_{j})_{S_{\theta_j}^*})$,
we observe that
\begin{equation}
\label{eq:eqExcessTheta}
\left(\w{\theta}_j- \theta^*_j\right)_{S_{\theta_j}^*} = \mathbb{H}_{S_{j}^{*}, S_{\theta_j}^*}^{-1} (Z_j)_{S_j^{*}} - \dfrac{n_k \kappa}{2} \mathbb{H}_{S_{\theta_j}^*, S_{\theta_j}^*}^{-1} {\rm sign}((\w{\theta}_{j})_{S_{\theta_j}^*}).
\end{equation}
Therefore, we set for each $j \in [M]$, 
\begin{equation}
\label{eq:eqDecompzComp}
z_{S_{\theta_j}^{*c}}  = -\dfrac{2}{n_k \kappa}\left(\mathbb{H}_{S_{\theta_j}^{*c}, S_{\theta_j}^*}  \mathbb{H}_{S_{j}^{*}, S_{\theta_j}^*}^{-1} (Z_j)_{S_j^{*}} - (Z_j)_{S_{\theta_j}^{*c}}\right) +  \mathbb{H}_{S_{\theta_j}^{*c}, S_{\theta_j}^*}  \mathbb{H}_{S_{j}^{*}, S_{\theta_j}^*}^{-1} {\rm sign}((\w{\theta}_{j})_{S_{\theta_j}^*}).
\end{equation}
We then have build an optimal solution $(\w{\theta}, \w{z})$ that satisfies the required condition.

\paragraph*{Second step.}

The goal of the second step is to prove that $\|\w{z}_{S^{*c}}\|_{\infty} < 1$ which implies the following result.
\begin{lemme}
Assume that $\|\w{z}_{S{^*c}}\|_{\infty} < 1$. Then, any solution $\widetilde{\theta}$ of the minimization problem
$\min_{\theta} \mathcal{C}(\theta)$    
satisfies $\widetilde{\theta}_{S^{*c}} = 0$.
\end{lemme}
\begin{proof}
Let $\widetilde{\theta}$ another solution. Then, it holds that
\begin{equation*}
R_{T,n}(\w{\theta}) +  \kappa \langle \w{z} , \w{\theta} \rangle  =    R_{T,n}(\widetilde{\theta}) + \kappa \sum_{j = 1}^M \sum_{j'=1}^M |\widetilde{\theta}_{j,j'}|,
\end{equation*}
we deduce that 
\begin{equation*}
R_{T,n}(\w{\theta}) -  \kappa \langle \w{z} , \widetilde{\theta}-\w{\theta} \rangle = 
R_{T,n}(\widetilde{\theta}) +  \kappa \left( \sum_{j = 1}^M \sum_{j'=1}^M |\widetilde{\theta}_{j,j'}| - \langle \w{z}, \widetilde{\theta} \rangle\right).
\end{equation*}
Since the pair $(\w{\theta},\w{z})$ satisfies Equation~\eqref{eq:eqSubGradientOptimal}, we have that
\begin{equation*}
 \kappa \w{z} = -\nabla R_{T,n}(\w{\theta}),   
\end{equation*}
which leads to 
\begin{equation*}
R_{T,n}(\w{\theta}) - R_{T,n}(\widetilde{\theta}) + \langle \nabla R_{T,n}(\w{\theta}), \widetilde{\theta}-\w{\theta} \rangle =   \kappa \left( \sum_{j = 1}^M \sum_{j'=1}^M |\widetilde{\theta}_{j,j'}| - \langle \w{z}, \widetilde{\theta} \rangle\right).
\end{equation*}
Hence, from the above equation and the convexity of $R_{T,n}$ we deduce that
\begin{equation*}
 \kappa \left( \sum_{j = 1}^M \sum_{j'=1}^M |\widetilde{\theta}_{j,j'}| - \langle \w{z}, \widetilde{\theta} \rangle\right) \leq 0.
\end{equation*}
Therefore, we obtain that
\begin{equation*}
\sum_{j = 1}^M \sum_{j'=1}^M |\widetilde{\theta}_{j,j'}| \leq   \langle \w{z}, \widetilde{\theta} \rangle = \sum_{j=1}^M\sum_{j'=1}^M \w{z}_{j,j'} \widetilde{\theta}_{j,j'}.
\end{equation*}
Since $\|\w{z}_{S^{*c}}\|_{\infty} < 1$, if there exists $\tilde{\theta}_{j,j'} \neq 0$ for 
$(j,j') \in S^{*c}$ we get
\begin{equation*}
 \sum_{j = 1}^M \sum_{j'=1}^M |\widetilde{\theta}_{j,j'}| <  \sum_{j = 1}^M \sum_{j'=1}^M |\widetilde{\theta}_{j,j'}|,   
\end{equation*}
which leads us to a contradiction. Therefore $\widetilde{\theta}_{S^{*c}} = 0$.
\end{proof}

Now we show that for $ \kappa\geq \dfrac{\log^4(nM^2)}{\sqrt{n}}$, we have $\|\w{z}_{S^{*c}}\|_{\infty} < 1$ on the event $\Omega_n$. From 
Equation~\eqref{eq:eqDecompzComp}, we deduce that for each $j \in [M]$
\begin{equation*}
\|\w{z}_{S_{\theta_j}^{*c}}\|_{\infty} \leq  \|\mathbb{H}_{S_{\theta_j}^{*c}, S_{\theta_j}^*}  \mathbb{H}_{S_{\theta_j}^*,S_{\theta_j}^*}^{-1}\|_{\infty} +  \|\mathbb{H}_{S_{\theta_j}^{*c}, S_{\theta_j}^*}  \mathbb{H}_{S_{\theta_j}^*,S_{\theta_j}^*}^{-1}\|_{\infty}
\dfrac{2}{n_k \kappa} \|(Z_j)_{S_{\theta_j}^*}\|_{\infty} + \dfrac{2}{n_k \kappa} \|(Z_j)_{S_{\theta_j}^{*c}}\|_{\infty}.
\end{equation*}
From Assumption (MI), we get for some $\gamma \in (0,1)$
\begin{equation}
\label{eq:eqBoundzhatComp}
\|\w{z}_{S^{*c}}\|_{\infty} \leq   (1-\gamma)\left(1 + \dfrac{2}{n_k \kappa} \|(Z_j)_{S_{\theta_j}^*}\|_{\infty}\right) +\dfrac{2}{n_k \kappa} \|(Z_j)_{S_{\theta_j}^{*c}}\|_{\infty}.
\end{equation}
From Lemma~\ref{lem:Z} we have with probability larger than $1-\dfrac{CM}{n}$ on an event $\Omega_n$
that
\begin{equation*}
\dfrac{1}{n_k} \|(Z_j)_{S_{\theta_j}^*}\|_{\infty} \leq \dfrac{C\log^3(nM^2)}{\sqrt{n}}, \;\;
\dfrac{1}{n_k} \|(Z_j)_{S_{\theta_j}^{*c}}\|_{\infty} \leq \dfrac{C\log^3(nM^2)}{\sqrt{n}}.
\end{equation*}
Hence, from Equation~\eqref{eq:eqBoundzhatComp}, for $n$ large enough, we deduce that, with probability larger than on $\Omega_n$,
\begin{equation*}
\|\w{z}_{S^{*c}}\|_{\infty} <1,
\end{equation*}
provided that $\dfrac{C\log^{3}(nM^2)}{ \kappa\sqrt{n}} \rightarrow 0$ as $n \rightarrow + \infty$.
Therefore, the choice $ \kappa \geq \dfrac{\log^4(nM^2)}{\sqrt{n}}$ yields the desired result.

\paragraph*{Third step.}
In the second step, we show for $n$ large enough that on $\Omega_n$, any solution of $\min_{\theta} \mathcal{C}(\theta)$
(with $\mathcal{C}$ given in \eqref{eq:penalisedcontrast}) is a solution of 
\begin{equation*}
\min_{\theta_{S^*}} \widetilde{R}_{T,n}(\theta_{S^*})+  \kappa \sum_{j = 1}^M \sum_{j' \in S_{\theta_j}^*} |\theta_{j,j'}|.  
\end{equation*}
In this step,  we establish the following result.
\begin{lemme}
\label{lem:lemBoundThetaInfty}
Let $\w{\theta}_{S^*}$ defined as 
\begin{equation*}
\w{\theta}_{S^*} \in \argmin{\theta_{S^*}}  \left\{\widetilde{R}_{T,n}(\theta_{S^*})+  \kappa \sum_{j = 1}^M \sum_{j' \in S_{\theta_j}^*} |\theta_{j,j'}|\right\}. 
\end{equation*}
Under Assumption~(ME), for $ \kappa  = \dfrac{\log^4(nM^2)}{\sqrt{n}}$, it holds that on $\Omega_n$
\begin{equation*}
\left\|\w{\theta}_{S^*} - \theta_{S^*}\right\|_{\infty}  \leq \dfrac{C\Lambda_0\max_j \sqrt{|S_{\theta_j}^*}|\log^4(nM^2)}{\sqrt{n}}. 
\end{equation*}
\end{lemme}

\begin{proof}
From Equation~\eqref{eq:eqExcessTheta}, we get for each $j \in \{1, \ldots,M\}$
\begin{equation*}
\left\|\w{\theta}_{S_{\theta_j}^*} - \theta_{S_{\theta_j}^*}\right\|_{\infty}  \leq 
\left\|\left(\dfrac{\mathbb{H}_{S_{j}^{*}, S_{\theta_j}^*}}{n_k}\right)^{-1} \dfrac{(Z_j)_{S_{\theta_j}^*}}{n_k}\right\|_{\infty} +
\dfrac{ \kappa}{2} \left\|\left(\dfrac{\mathbb{H}_{S_{\theta_j}^*, S_{\theta_j}^*}}{n_k}\right)^{-1}{\rm sign}((\w{\theta}_{j})_{S_{\theta_j}^*})\right\|_{\infty}.
\end{equation*}
Applying Lemma~\ref{lem:boundInftyMatrix}, and~\ref{lem:Z} together with Assumption~(ME), we obtain 
\begin{equation*}
\left\|\w{\theta}_{S_{\theta_j}^*} - \theta^*_{S_{\theta_j}^*}\right\|_{\infty} \leq \Lambda_0 \sqrt{\left|S_{\theta_j}^*\right|}\left(\dfrac{C\log^3(nM^2)}{\sqrt{n}} +  \kappa\right). 
\end{equation*}
Therefore, the choice of $ \kappa = \dfrac{\log^4(nM^2)}{\sqrt{n}}$ yields 
the desired result.
\end{proof}

\paragraph*{Fourth step.}

We deduce from Lemma~\ref{lem:lemBoundThetaInfty} and Assumption~\ref{ass:MS} that
\begin{equation*}
{\rm sign}(\w{\theta}_{S^*}) = {\rm sign}(\theta^*_{S^*}).
\end{equation*}
Therefore, from Equation~\eqref{eq:eqExcessTheta}, we deduce that on $\Omega_n$ for $ \kappa = \dfrac{\log^4(nM^2)}{\sqrt{n}}$ ,
\begin{equation*}
\theta_{S^*} \mapsto \min_{\theta_{S^*}} \widetilde{R}_{T,n}(\theta_{S^*})+  \kappa \sum_{j = 1}^M \sum_{j' \in S_{\theta_j}^*} |\theta_{j,j'}|, 
\end{equation*}
admits a unique minimizer $\w{\theta}_{S^*}$ which satisfies for each $j \in\{1, \ldots,M\}$,
\begin{equation*}
(\w{\theta}_j)_{S_{\theta_j}^*} = (\theta_j)_{S_{\theta_j}^*} +
\mathbb{H}_{S_{\theta_j}^*, S_{\theta_j}^*}^{-1} (Z_j)_{S_{\theta_j}^*} - \dfrac{n_k \kappa}{2} \mathbb{H}_{S_{\theta_j}^*, S_{\theta_j}^*}^{-1} {\rm sign}(({\theta^*}_{j})_{S_{\theta_j}^*}).
\end{equation*}

Hence, in view of Steps~2, with the choice of $ \kappa =  \dfrac{\log^4(nM^2)}{\sqrt{n}}$, we then have shown that there is a unique solution $\w{\theta}$ of $\min_{\theta} \mathcal{C}(\theta)$ which satisfies on $\Omega_n$
\begin{equation*}
\w{\theta}_{S^{*c}} = 0, \;\; {\rm and} \;\; {\rm sign}(\w{\theta}_{S^*}) = {\rm sign}(\theta^*_{S^*}),
\end{equation*}
and
\begin{equation*}
\left\|\w{\theta}_{S^*} - \theta_{S^*}\right\|_{\infty}  \leq \dfrac{C\Lambda_0\max_j \sqrt{|S_{\theta_j}^*|}\log^4(nM^2)}{\sqrt{n}}. 
\end{equation*}

\section{Proofs for the rate of convergence of \texttt{ERMLR} algorithm}
\label{app:RatesofCve}

\setcounter{lemme}{0}
\renewcommand{\thelemme}{D.\arabic{lemme}}
\setcounter{prop}{0}
\renewcommand{\theprop}{D.\arabic{prop}}

We first establish a technical result in Section~\ref{subsec:classifResultTechRes}, then rate of convergence of the \texttt{ERMLR} algorithm is given in Section~\ref{subsec:classifResultProof}. 

\subsection{Technical result}
\label{subsec:classifResultTechRes}


We recall that the set $\Theta_n$ is defined as follows

\begin{equation*}
\Theta_n := \left\{\theta = (\mu, A) \in \mathbb{R}_{+}^{M}\times \mathbb{R}_{+}^{M^2}, \mu_j \in \left[\dfrac{1}{n}, \log(n)\right], j \in [M], \left\|A\right\|_F \leq n \right\}.
\end{equation*}
We also introduce the set $\Pi$ of conditional probabilities  

\begin{equation*}
{{\Pi}} := \left\{\pi_{{p },\theta}= 
\left(\frac{ {p}_k{\rm e}^{F_{\theta_k}(\cdot)}}{\sum_{k'=1}^K {p}_{k'} {\rm e}^{F_{\theta_{k'}}(\cdot)}}\right)_{k \in [K]}
: \; \theta =(\theta_1, \ldots, \theta_K) \in {\Theta}^K_n, ~\sum_{k=1}^K p_k = 1, ~\min_k {p_k}> \frac{p_0}{2}
\right\}
\end{equation*}

The following result provides a bound on $\ell_1$-distance between two elements of the set $\Pi$. It shows that this distance can be bounded by the distance between the corresponding parameters of the associated model.

\begin{prop}\label{prop:distPi}
Let $\pi = \pi_{p,\theta}$ and $\pi' = \pi_{p^{'},\theta^{'}}$ two elements of $\Pi$. Grant Assumptions~\ref{ass:h}, \ref{ass:mu} and \ref{ass:prob}, 
the following holds
\begin{eqnarray*}
\mathbb{E}\big[ \left\|\pi -  \pi'\right\|_1\big] &\leq & \frac{K}{p_0} \|{{p}-{p}'}\|_{1} + CK^2n^2\log(n)\left(\sqrt{M} \max_{k\in [K]}\|\mu_k-\mu_k'\|_1 + M  \max_{k\in [K]}\|A_k- A_k'\|_F\right),
\end{eqnarray*}
where $C$ is a constant depending on $T$, $\mu_0$, $\mu_1$ and $\|h\|_\infty$.  
\end{prop}

\begin{proof}
Let us consider $\pi, \pi' \in \Pi$  with respective parameters $( p, \theta)$, and $( p', \theta')$. We have that
\begin{eqnarray}
\label{eq:eqdistPiFirstDecomp}
\left\|\pi(\cT)-\pi'(\cT)\right\|_1 &\leq &\left\|\pi(\cT) - \pi_{{p},\theta'}(\cT)\right\|_1 + \left\|\pi_{{p},\theta'}(\cT)- \pi'(\cT)\right\|_1.
\end{eqnarray}  
Since for any $k$, $j$ and $(x_1, \ldots, x_K)$, 
\begin{equation*}
\left| \dfrac{\partial \phi^{{p}}_k(x_1, \ldots,x_K)}{\partial p_j}\right| \leq  \dfrac{1}{p_0},
\end{equation*}
we deduce by mean value inequality
\begin{equation*}
\left\|\pi_{{p},\theta'}(\cT)- \pi'(\cT)\right\|_1 \leq \dfrac{K}{p_0} \left\|{p}-{p}'\right\|_1.
\end{equation*}
Besides for any $k$, $j$ and ${p}$,
\begin{equation*}
\left| \dfrac{\partial \phi^{{p}}_k(x_1, \ldots,x_K)}{\partial x_j}\right| \leq  1,
\end{equation*}
we also deduce
$$\left\|\pi(\cT) - \boldsymbol\pi_{{p},\theta'}(\cT)\right\|_1 
\leq K\sum_{k=1}^K \left| F_{{\theta_k}}(\cT)-F_{{\theta_k'}}(\cT) \right|.
$$
Therefore, from Equation~\eqref{eq:eqdistPiFirstDecomp}, we obtain
$$
\mathbb{E} \left[\left\|\pi(\cT)-\pi'(\cT)\right\|_1\right] \leq
 \dfrac{K}{p_0} \left\|{p}-{p}'\right\|_1 
 + K\sum_{k=1}^K \mathbb{E}\left[\left| F_{\theta_k}(\cT)-F_{\theta_k'}(\cT) \right|\right].
$$
Hence, it remains to bound the second term in the {\it r.h.s.} of the above inequality.
Using Cauchy-Schwartz inequality, for each $k$, we have that
\begin{align}
\label{eq:eqdecompFdistPi}
& \mathbb{E}\left[\left| F_{\theta_k}(\cT)-F_{\theta_k'}(\cT) \right|\right] \nonumber\\
& =  \E\left[ \left| \sum_{j=1}^M\left(\int_0^T \log\left( \frac{\lambda_{j,\theta_k}(t)}{\lambda_{j,{\theta_k'}}(t)} \right) \dd N_j(t) - \int_0^T\left(\lambda_{j,{\theta_k}}(t)- \lambda_{j,{\theta_k'}}(t)\right) \dd t \right)\right| \right]\nonumber\\
 &\leq   \E\left[ \left(\sum_{j=1}^M\int_0^T \left|\log\left( \frac{\lambda_{j,{\theta_k}}(t)}{\lambda_{j,\theta'_{k}}(t)} \right)\right| \dd N_j(t)\right)^2 \right]^{1/2} + \E\left[\sum_{j=1}^M\int_0^T\left|\lambda_{j,{\theta_k}}(t) -\lambda_{j,\theta'_{k}}(t)\right|  \dd t  \right].
\end{align}
Now, we observe that
\begin{equation*}
\left|\lambda_{j,\theta_k}(t) - \lambda_{j,{\theta_k'}}(t)\right| \leq |\mu_{k,j}-\mu'_{k,j}| +\|h\|_\infty\sum_{j'=1}^M |a_{k,j,j'}-a'_{k,j,j'} |N_{j'}(T). 
\end{equation*}
Therefore, we deduce
\begin{equation}
\label{eq:eqDecompPiEq1}
\E\left[\sum_{j=1}^M\int_0^T\left|\lambda_{j,{\theta_k}}(t) -\lambda_{j,\theta'_{k}}(t)\right|  \dd t  \right] 
\leq T\sum_{j=1}^M|\mu_{k,j}-\mu'_{k,j}| 
+  T \|{h}\|_\infty \sum_{j'=1}^M \sum_{j=1}^M |a_{k,j,j'}-a'_{k,j,j'}| \E\left[N_{j'}(T)\right].
\end{equation}

Now, we bound the first term in the {\it r.h.s.} of Equation~\eqref{eq:eqdecompFdistPi}.
Using that $x \mapsto \log(1+x)$ is a Lipschitz function, we obtain: 
\begin{align}\label{eq:eqDecompPiEq2}
\left|\log\left( \frac{\lambda_{j,{\theta_k}}(t)}{\lambda_{j,\theta'_{k}}(t)} \right)\right| 
& \leq  
 \left|\log\left(\frac{\mu_{k,j}}{\mu'_{k,j}}\right)\right|
  + \left|\frac{\lambda_{j,{\theta_k}}(t)}{\mu'_{k,j}} - \frac{\lambda_{j,{\theta_k'}}(t)}{\mu_{k,j}} \right| \nonumber\\ 
& \leq  n\left|\mu_{k,j} - \mu'_{k,j}\right| + n^2\left|\mu_{k,j}\lambda_{j,{\theta_k}}(t)- \mu'_{k,j}\lambda_{j,{\theta_k'}}(t) \right|  \nonumber\\ 
& \leq    n\left|\mu_{k,j} - \mu'_{k,j}\right| + n^2 \left(|\mu_{k,j} -\mu'_{k,j}|\lambda_{j,\theta'_{k}}(t) + \mu_n' \left|\lambda_{j,{\theta_k}}(t)-\lambda_{j,\theta'_{k}}(t) \right|\right) \nonumber\\ 
& \leq    n\left|\mu_{k,j} - \mu'_{k,j}\right| + n^2 \Bigg(|\mu_{k,j} -\mu'_{k,j}|\lambda_{j,\theta'_{k}}(t) \nonumber\\
& \hspace{1em} + \log(n) \Big(|\mu'_{k,j}-\mu_{k,j}| + \|h\|_{\infty} \sum_{j'=1}^M N_{j'}(T)|a_{k,j,j'}-a'_{k,j,j'}|\Big)\Bigg).
\end{align}

Besides, applying the Doob's decomposition for the processes $N_j, j \in [M]$, and the Cauchy-Schwartz's inequality, we get  

\begin{eqnarray}\label{eq:eqDecompPiEq3}
\E\left[ \left(\sum_{j=1}^M\int_0^T \left|\log\left( \frac{\lambda_{j,{\theta_k}}(t)}{\lambda_{j,{\theta_k'}}(t)} \right)\right| \dd N_j(t)\right)^2 \right] &\leq &  M \sum_{j=1}^M\E\left[  \int_0^T \log^2\left( \frac{\lambda_{j,{\theta_k}}(t)}{\lambda_{j,\theta'_{k}}(t)} \right) \lambda^{*}_{Y,j}(t)  \dd t \right] \nonumber\\
&+&  M \sum_{j=1}^M \E\left[ \left( \int_0^T\left| \log\left( \frac{\lambda_{j,{\theta_k}}(t)}{\lambda_{j,{\theta_k'}}(t)} \right) \right|
 \lambda^{*}_{Y,j}(t)  \dd t \right)^2 \right].
\end{eqnarray}

From Assumption~\ref{ass:expomoments}, we have $\E\left[\left( \lambda^{*}_{Y,j}(t)\right)^2\right]<\infty$. Therefore, the first term in the {\it r.h.s.} in Equation~\eqref{eq:eqDecompPiEq3} can be bounded as follows
\begin{eqnarray*}
\E\left[  \int_0^T \log^2\left( \frac{\lambda_{j,\theta_{k}}(t)}{\lambda_{j,\theta'_{k}}(t)} \right) \lambda^{*}_{Y,j}(t) \dd t \right] 
&\leq & \int_0^T \E \left[ \log^4\left( \frac{\lambda_{j,\theta_{k}}(t)}{\lambda_{j,\theta'_{k}}(t)} \right)  \right]^{1/2} \E\left[\left( \lambda^{*}_{Y,j}(t)\right)^2\right]^{1/2}  \dd t\\
& \leq & C T \sup_{t \in [0,T]} \E \left[ \log^4\left( \frac{\lambda_{j,\theta_{k}}(t)}{\lambda_{j,\theta'_{k}}(t)} \right)  \right]^{1/2}.
\end{eqnarray*}
Similarly, we obtain:
\begin{eqnarray*}
 \E\left[ \left( \int_0^T\left| \log\left( \frac{\lambda_{j,\theta_{k}}(t)}{\lambda_{j,\theta'_{k}}(t)} \right) \right|
 \lambda^{*}_{Y,j}(t) \dd t \right)^2 \right]
  &\leq & T \E \left[\int_0^T \log^2\left( \frac{\lambda_{j,\theta_{k}}(t)}{\lambda_{j,\theta'_{k}}(t)} \right)  \left( \lambda^{*}_{Y,j}(t)\right)^2 \dd t \right] \\
 &\leq & C T^2 \sup_{t \in [0,T]} \E \left[ \log^4\left( \frac{\lambda_{j,\theta_{k}}(t)}{\lambda_{j,\theta'_{k}}(t)} \right)  \right]^{1/2}.
\end{eqnarray*}
Then, by Assumption~\ref{ass:h}, from Equation~\eqref{eq:eqDecompPiEq2} and Equation~\eqref{eq:eqDecompPiEq3}, we get 

\begin{align*}
 & \E\left[ \left(\sum_{j=1}^M\int_0^T \left|\log\left( \frac{\lambda_{j,\theta_{k}}(t)}{\lambda_{j,\theta'_{k}}(t)} \right)\right| \dd N_j(t)\right)^2 \right] \\
 & \leq   C T^2 M\sum_{j=1}^M \sup_{t \in [0;T]} \E \left[
 |\mu_{k,j}-\mu_{k,j}'|^4 \left( n+n^2\lambda_{j,\theta_{k}'}(t)+ n^2\log(n)\right)^4 \right.\\
 &\hspace{1em} \left. + n^8\log(n)^{4} \|h\|_\infty^4 \left( \sum_{j'=1}^M N_{j'}(T) |a_{k,j,j'}-a'_{k,j,j'}|\right)^4
 \right]^{1/2}\\
  & \leq   C T^2 M\sum_{j=1}^M \left(\left[ n^4\sup_{t \in [0;T]} \E \left[(\lambda_{j,\theta_{k}'}(t))^4\right]^{1/2}+ n^2+ n^4\log(n)^2\right]
 |\mu_{k,j}-\mu_{k,j}'|^2 \right.\\
 &\hspace{1em} \left.+ C n^4\log(n)^{2}\E\left[ \left( \sum_{j'=1}^M N_{j'}(T) |a_{k,j,j'}-a'_{k,j,j'}|\right)^4\right]^{1/2}\right) \\
\end{align*} 
where $C$ is a constant depending on $\mu_0$, $\mu_1$ and $\|h\|_\infty$. 
In view of Assumption~\ref{ass:expomoments} $\E \left[(\lambda_{j,\theta_{k}'}(t))^4\right] \leq C$. Therefore, from the above equation, and Cauchy Schwartz's inequality, we deduce
\begin{multline*}
E\left[ \left(\sum_{j=1}^M\int_0^T \left|\log\left( \frac{\lambda_{j,\theta_{k}}(t)}{\lambda_{j,\theta'_{k}}(t)} \right)\right| \dd N_j(t)\right)^2 \right]  \leq 
CT^2 M \left(n^4+n^2+n^4\log(n)^{2}\right) \sum_{j=1}^M |\mu_{k,j}-\mu_{k,j}'|^2 \\
+ CT^2M n^4\log(n)^{2}  \E\left[ \left( \sum_{j'=1}^M N_{j'}(T) ^2\right)^2\right]^{1/2} \sum_{j=1}^M\sum_{j'=1}^M|a_{k,j,j'}-a'_{k,j,j'}|^2.
\end{multline*}
From Assumption~\ref{ass:expomoments}
we have that $ \E\left[ \left( \sum_{j'=1}^M N_{j'}(T) ^2\right)^2\right] \leq CM^2$.
Thus, gathering Equations \eqref{eq:eqdecompFdistPi} and \eqref{eq:eqDecompPiEq1}, it comes  
\begin{eqnarray*}
\E[\|\pi-\pi'\|_1] & \leq & \frac{K}{p_0}\|p-p'\|_1 \\
&&+ K C\sum_{k=1}^K \left( \sum_{j=1}^M |\mu_{k,j}-\mu_{k,j}'|+\sum_{j=1}^M\sum_{j'=1}^M|a_{k,j,j'}-a'_{k,j,j'}| \right)\\
&& + K C n^2\log(n) \sum_{k=1}^K \left( M  \sum_{j=1}^M |\mu_{k,j}-\mu_{k,j}'|^2+M^{2}\sum_{j=1}^M\sum_{j'=1}^M|a_{k,j,j'}-a'_{k,j,j'}|^2 \right)^{1/2}
\end{eqnarray*}
with $C$ depending on $\mu_0$, $\mu_1$, $\|h\|_\infty$ and $T$.
Finally, using that $\|x\|_2 \leq \|x\|_1 \leq \sqrt{d} \|x\|_2$ for $x \in \R^d$, we obtain
\begin{eqnarray*}
\E\big[\|\pi-\pi'\|_1\big] & \leq & \frac{K}{p_0}\|p-p'\|_1 \\
&&+ K^2 Cn^2\log(n)^2\max_{k\in [K]}\left( \sum_{j=1}^M |\mu_{k,j}-\mu_{k,j}'|+ \ \sum_{j=1}^M\sum_{j'=1}^M|a_{k,j,j'}-a'_{k,j,j'}|\right) \\
&& + K^2 C n^2\log(n)^2 \max_{k\in [K]}\left(\sqrt{M} \sum_{j=1}^M |\mu_{k,j}-\mu_{k,j}'|+M  \|A_k- A_k'\|_F\right)
\end{eqnarray*}
thus
\begin{eqnarray*}
\E\big[\|\pi-\pi'\|_1\big] 
&\leq & \frac{K}{p_0}\|p-p'\|_1 + K^2 C n^2\log(n)\sqrt{M}\max_{k\in [K]}\|\mu_k-\mu_k'\|_1 \\
&&+ K^2 C M n^2\log(n)\max_{k\in [K]}\|A_k- A_k'\|_F.
\end{eqnarray*}

Finally, combining the above equation, Equations~\eqref{eq:eqdecompFdistPi} and~\eqref{eq:eqDecompPiEq1} yields the desired result.


%
%
\end{proof}


\subsection{Proof of Theorem~\ref{thm:riskERMsupport}}
\label{subsec:classifResultProof}

We begin this section by a lemma that provides a bound on the $\varepsilon$-covering number of the set $\hat{\Theta}$ defined in Equation~\ref{eq:thetaHatSet}.

\begin{lemme}\label{lem:epsnet}
Let $\varepsilon >0$.
There exists an $\varepsilon$-net $\mathcal{M}_{\varepsilon} \subset {\hat{\Theta}}$
with
\begin{equation*}
|{\cM}_{\varepsilon}|  \leq \left( \frac{(\log(n)-1/n)M}{\varepsilon}\right)^{MK} \left( \frac{3 n}{\varepsilon}\right)^{  \sum_{k}\w{S}_k}. 
\end{equation*}
In particular, for all $(\mu,A) \in \w{\Theta}$  there exists $(\mu_{\varepsilon}, A_{\varepsilon}) \in \mathcal{M}_{\varepsilon}$ s.t. $\max_{k\in [K]}\|\mu_k-\mu_{k,\varepsilon}\|_1 \leq \varepsilon$ and $\|A_k-A_{k,\varepsilon}\|_F \leq \varepsilon$. 
\end{lemme}


\begin{proof}[Proof of Lemma \ref{lem:epsnet}]
First, we observe that the set
$$\left\{ \frac{1}{n}+ k\frac{(\log(n)-1/n)}{ \lceil \frac{M(\log(n)-1/n)}{\varepsilon}  \rceil}, ~ k \in \left\{1, \ldots,  \frac{M(\log(n)-1/n)}{\varepsilon}-1 \right\}  \right\}$$
is and $\varepsilon/M$-cover of the interval $[1/n \log(n)]$. 
Therefore, we deduce that there exists
$\mathcal{M}_{\varepsilon, \mu}$ an $\varepsilon$-cover of 
$\{\mu \in  \R^M, ~s.t. ~\mu \in \Theta_n \}$ for $\|\cdot\|_1$, such that 
\begin{equation}{\label{eq:epsnetmu}}
\mathcal{M}_{\varepsilon, \mu} \leq \left( \frac{(\log(n)-1/n)M}{\varepsilon}\right)^{M}.
\end{equation}

Let $k \in [K]$. 
For $\varepsilon>0$, the covering number of the Euclidean ball centered in 0 and with radius $n$ in $\R^{\w{S}_k}$, satisfies
$$\mathcal{N}(\varepsilon, \mathcal{\bar{B}}(0, n), \|.\|_2) \leq \left( \frac{3n}{\varepsilon}\right)^{\w{S}_k}.$$ 
Hence, we deduce that there exists $\mathcal{M}_{\varepsilon, A,k}$ an $\varepsilon$-cover of $\{A \in \Theta_n, ~s.t. ~\text{supp}(A)= \w{S}_k \}$, for $\|\cdot\|_F$, such that
\begin{equation}{\label{eq:epsnetA}}
\mathcal{M}_{\varepsilon, A,k} \leq \left(\frac{3n}{\varepsilon} \right)^{\w{S}_k}.
\end{equation}

From Equation \eqref{eq:epsnetmu} and \eqref{eq:epsnetA} we obtain the desired result. 
\end{proof}

\begin{proof}[Proof of Theorem~\ref{thm:riskERMsupport}]
We first recall that the construction of the \texttt{ERMLR} algorithm is based on a dataset 
$\mathcal{D}_n = \{(\cT_T^{(i)}, Y^{(i)}), i = 1, \ldots, 2n\}$ of size $2n$ which is split into two independent dataset of same size $n$ that are denoted respectively $\mathcal{D}_n^{(1)}$ and $\mathcal{D}_n^{(2)}$.

Based on the first sample $\mathcal{D}_n^{(1)}$, we estimate the vector of weights $p^*$ by its empirical frequencies $\hat{p}$. Hence for each $k$, we have 
\begin{equation*}
\hat{p}_k = \dfrac{1}{n} \sum_{i = 1}^n \one_{\{Y^{(i)} = k\}}  
\end{equation*}
Then, based on sample $\mathcal{D}_n^{(2)}$, we build the estimator $\w{S} := \left(\w{S}_1, \dots, \w{S}_K\right)$ as described in Section~\ref{subsec:estimationSupport}.
Besides, we also build the estimator of the vector of score function $\hat{f} = f_{\hat{\theta}^{R}}$, and $\hat{g}$ its associated classifier.
Since $\mathcal{D}_n^{(1)}$ and $\mathcal{D}_n^{(2)}$ are independent, we have that 
$\hat{p}$ is independent on $\w{f}$ and $\w{g}$.

Let us introduce the set $\mathcal{A} = \left\{{ \w{p}} :\; \min({\w{p}}) \geq \frac{p_0}{2}\right\}$. Note that on $\mathcal{A}^c$ we have 
\begin{equation*}
|\min({p^*}) -  \min({ \w{p}})| \geq \frac{p_0}{2}, 
\end{equation*}
which implies that there exists $k\in\mathcal{Y}$ s.t. $|p^*_k - \w{p}_k| \geq \frac{p_0}{2}$. Thus, using Hoeffding's inequality we get
\begin{eqnarray}\label{eq:hoeffding1}
\mathbb{P}(\mathcal{A}^c) &\leq & \sum_{k=1}^{K} \mathbb{P}\left(|p^*_k - \w{p}_k| \geq \frac{p_0}{2}\right) \nonumber\\
    & \leq& 2K {\rm e}^{-np_0^2/2}.
\end{eqnarray}



Now, let us work on $\Omega = \mathcal{A}\bigcap \{\hat{S} = S^*\}$, and denote 
\begin{equation}\label{eq:deltan}
\Delta_n :=  \sum_{k=1}^K (\w{p}_k-p_k^*)^2,
\end{equation}
which is a random variable independent from $\cD^{(2)}_{n}$.
We also recall that for each $\theta \in \hat{\Theta}$,  the score function $f_{\theta}$ is defined as follows
\begin{equation*}
f_{\theta} (\mathcal{T}_T) = 2\pi_{k,\w{p}, \theta}(\mathcal{T}_T) -1, ~k \in [K].
\end{equation*}
We introduce
\begin{equation*}
\tilde{\theta} = \argmin{\theta \in \hat{\Theta}} \mathcal{R}_2(f_{\theta}).   
\end{equation*}
The oracle counterpart of $\hat{f}$.
Our aim is to control
\begin{equation}
\label{eq:eqDecompRisk1}
\E\left[ \cR_2(\w{f})- \cR_2({f}^*)\right] = \E\left[ \left(\cR_2(\w{f})- \cR_2({f}^*) \right)\one_{\{\Omega\}}\right]+ \E \left[ \left(\cR_2(\w{f})- \cR_2({f}^*) \right)\one_{\{\Omega^{c}\}}\right].
\end{equation}
Since for each $\theta \in \hat{\Theta}$ defined by~\eqref{eq:thetaHatSet}, $\cR_2(f_{\theta})$ is bounded, from Theorem~\ref{prop:convsupport}, and Equation~\eqref{eq:hoeffding1}, we deduce 
that
\begin{equation}
\label{eq:eqDecompRisk3}
\E[ \left(\cR_2(\w{f})- \cR_2({f}^*) \right)\one_{\{\Omega^{c}\}}] \leq C \mathbb{P}\left(\Omega^{c}\right)  \leq C \left(\dfrac{1}{n}+ \exp\left(-np_0^2/2\right) \right).
\end{equation}
Therefore, it remains to bound the first term in the {\it r.h.s.} of Equation~\eqref{eq:eqDecompRisk1}. Hence, we work on the set $\Omega$.
We consider the following decomposition
\begin{equation}
\label{eq:eqDecompRisk2}
\mathcal{R}_2(\hat{f}) -  \mathcal{R}_2(f^*)  =  \left(\mathcal{R}_2(\hat{f}) - \mathcal{R}_2(f_{\tilde{\theta}})\right) + \left(\mathcal{R}_2(f_{\tilde{\theta}}) -\mathcal{R}_2(f^*) \right)   
\end{equation}

In a first step, we control the second term in the {\it r.h.s.} of the above equation.
For $n$ large enough, we observe that on $\Omega$, $\theta^* \in \hat{\Theta}$.
Therefore, from the definition of $\tilde{\theta}$, we deduce

\begin{eqnarray*}
\mathcal{R}_2(f_{\tilde{\theta}}) -\mathcal{R}_2(f^*) & = & \mathcal{R}_2(f_{\tilde{\theta}})-\mathcal{R}_2(f_{\theta^*})+\mathcal{R}_2(f_{\theta^*})- \mathcal{R}_2(f^*)\\
& \leq & \mathcal{R}_2(f_{\theta^*})- \mathcal{R}_2(f^*).
\end{eqnarray*}
Then on $\Omega$, we deduce from the mean value theorem that
\begin{equation}
\label{eq:eqDecompRisk4}
\mathcal{R}_2(f_{\tilde{\theta}}) -\mathcal{R}_2(f^*) \leq   \mathcal{R}_2(f_{\theta^*})- \mathcal{R}_2(f^*) \leq C \Delta_n,
\end{equation}
with $\Delta_n$ given in Equation \eqref{eq:deltan}.
Since, $\mathbb{E}\left[\Delta_n\right] \leq \dfrac{C}{n}$, from Equation~\eqref{eq:eqDecompRisk2}, we deduce that
\begin{equation}
\label{eq:eqDecompRisk4bis}
\E\left[ \left(\cR_2(\w{f})- \cR_2({f}^*) \right)\one_{\{\Omega\}}\right] \leq \E\left[\mathcal{R}_2(\hat{f}) - \mathcal{R}_2(f_{\tilde{\theta}}) \one_{\{\Omega\}}\right] + \dfrac{C}{n}.  
\end{equation}
Now, we focus on the first term in the {r.h.s.} of Equation~\eqref{eq:eqDecompRisk2}.
We denote 
$$D_{{f}}:=\cR_2({f})-\cR_2(f_{\tilde{\theta}}), \quad \mbox{and} \quad \w{D}_{{f}}:=\w{\cR}_2({f})-\w{\cR}_2(f_{\tilde{\theta}}).$$
And we want to control $\E[D_{\w{f}}].$
By Lemma \ref{lem:epsnet}, there exists a subset ${\cM_{\varepsilon}} \subset \hat{\Theta}$ such that for $\hat{\theta}^{R} = \left(\hat{\mu}, \hat{A}\right)$, there exists $\theta_{\varepsilon} = (\mu_{\varepsilon},{ A}_{\varepsilon}) \in {\cM}_{\varepsilon}$
satisfying
\begin{equation*}
\max_{k\in [K]}\|\mu_{k,\varepsilon}-\hat{\mu}_k\|_1 \leq \varepsilon \quad {\rm and} \quad \max_{k\in [K]}\|A_{k,\varepsilon} - \hat{A}_k\|_{F} \leq \varepsilon.
\end{equation*}
Then, the following decomposition holds
\begin{eqnarray*}
D_{\w{f}} &\leq & D_{\w{f}}- 2 \w{D}_{\w{f}}\\
&=& (D_{\w{f}}-D_{{f}_{\theta_{\varepsilon}}})+ (2\w{D}_{{f}_{\theta_{\varepsilon}}}-2\w{D}_{\w{f}}) + (D_{{f}_{\theta_{\varepsilon}}}-2\w{D}_{{f}_{\theta_{\varepsilon}}})\\
&=:& T_1+ T_2 +T_3. 
\end{eqnarray*}
Applying Proposition~\ref{prop:distPi}  with $\varepsilon= 1/(n^3 M \log(n))$  we get
\begin{equation*}
\E\left[T_i\right] \leq \frac{C}{n}, \quad \mbox{for } \, i=1,2.
\end{equation*}
Besides, 
$$T_3 \leq \max_{ \theta \in {\cM_{\varepsilon}} } (D_{{f_{\theta}}}-2\w{D}_{{f_{\theta}}}).$$
Therefore, gathering Equation~\eqref{eq:eqDecompRisk1}, ~\eqref{eq:eqDecompRisk3}, ~\eqref{eq:eqDecompRisk4}, and~\eqref{eq:eqDecompRisk4bis},
we deduce that 
\begin{eqnarray}
\label{eq:bound2}
\E[ \cR_2(\w{f})- \cR_2({f}^*)] &\leq & \E\left[\max_{ \theta \in {\cM_{\varepsilon}} } (D_{{f}_{\theta}}-2\w{D}_{{f}_{\theta}})\one_{\{\Omega\}}\right] +  \frac{C}{n}. 
\end{eqnarray}

To finish the proof, it remains to control the first term in the \textit{r.h.s.} of Inequality~\eqref{eq:bound2}.
Conditional on $\mathcal{D}_n^{(1)}$,
we have that 
\begin{equation*}
 \mathbb{E}\left[\max_{ \theta \in {\cM_{\varepsilon}} } (D_{{f}_{\theta}}-2\w{D}_{{f}_{\theta}})\one_{\{\Omega\}} | \mathcal{D}_n^{(1)}\right]  = \one_{\{\cA\} }\mathbb{E}\left[\max_{ \theta \in {\cM_{\varepsilon}} } (D_{{f}_{\theta}}-2\w{D}_{{f}_{\theta}})\one_{\{\hat{S} = S^{*}\}} | \mathcal{D}_n^{(1)}\right].  
\end{equation*}
Note that On the set $\{\hat{S} = S^*\}$,
the set $\mathcal{M}_{\varepsilon}$ is an $\varepsilon$-net of the {\it deterministic} set 
\begin{equation*}
\tilde{\Theta} =  \left\{\theta = \left(\theta_1, \ldots, \theta_K\right) \in \Theta_n^K, \;\; {\rm supp}(A_k)= S^*_k\right\},
\end{equation*}
and then is also deterministic. Besides, from Lemma~\ref{lem:epsnet}, we deduce that for 
$\varepsilon = \dfrac{1}{n^3M\log(n)}$
\begin{equation*}
\log\left(\left|\mathcal{M}_{\varepsilon}\right|\right) \leq CK\left(M+s^*\right) \dfrac{\log(nM)}{n}.    
\end{equation*}
Furthermore, for $u> 0$ conditional on $\mathcal{D}_n^{(1)}$, it holds that
\begin{multline}
\label{eq:eqboundExpInt}
\E \left[\max_{\theta \in {\cM}_{\varepsilon}} (D_{{f}_{\theta}}-2\w{D}_{{f}_{\theta}}) \one_{\{\hat{S} = S^*\}}\right]
 \leq  u+ \int_u^{\infty} \P\left( \max_{\theta \in {\cM}_{\varepsilon}} (D_{{f}_{\theta}}-2\w{D}_{{f}_{\theta}})\one_{\{\hat{S} = S^*\}} \geq t \right)\dd t \\
 \leq u+ \int_u^{\infty} \P\left(\max_{\theta \in {\cM}_{\varepsilon}} (D_{{f}_{\theta}}-2\w{D}_{{f}_{\theta}}) \geq t \right)\dd t.
\end{multline}
Now, we have to bound the last term in the above equation. Let $\theta \in \mathcal{M}_{\varepsilon}$, and $f:=f_{\theta}$.
Let us introduce the least squares function 
$$\ell_{f}(Z, \cT):= \sum_{k = 1}^K (Z_k-{{f}}^k(\cT))^2.$$ 
Since for each $\theta \in \tilde{\Theta}$, ${f}_{\theta}$ is uniformly bounded by $1$, we get from Bernstein's inequality that, conditionally on $\cD_n^{(1)}$, for $t \geq 0$
\begin{eqnarray}
\label{eq:eqBernstein}
\P\left( D_{{f}}-2\w{D}_{{f}} \geq t \right) &\leq & \P\left( 2(D_{{f}}-\w{D}_{{f}}) \geq t + D_{{f}} \right)\nonumber \\
&\leq & \exp \left( \frac{-n(t+D_{f})^2/8}{ B_{{f}} + (t+D_{f})4K/3} \right),
\end{eqnarray}
with
\begin{equation*}
B_{{f}} :=  \E\left[\left(\ell_{{f}}(Z, \cT)-\ell_{{f}_{\tilde{\theta}}}(Z, \cT)\right)^2\right].
\end{equation*}
From the Cauchy-Schwartz inequality, we observe that conditionally on $\mathcal{D}_n^{(1)}$
\begin{eqnarray*}
\E\left[\left(\ell_{{f}}(Z, \cT)-\ell_{{f}^*}(Z, \cT)\right)^2\right]
&\leq &C_K 
\sum_{k=1}^K \mathbb{E}\left[({f}^k(\cT)-{f}^{*k}(\cT))^2\right] \\
&=& C_K \left(\cR_2({f})-\cR_2({f}^*) \right).
\end{eqnarray*}
Thus, since
\begin{equation*}
B_{{f}} \leq  2 \E\left[ \Big(\ell_{f}(Z, \cT)-\ell_{{f}^*}(Z, \cT)\Big)^2 \right] + 2\E\left[ \left(\ell_{{f}_{\tilde{\theta}}}(Z, \cT)-\ell_{{f}^*}(Z, \cT)\right)^2\right],    
\end{equation*}
we deduce that
$$B_{{f}}  \leq C_K\left(\cR_2({f})-\cR_2({f}^*)+\cR_2({f}_{\tilde{\theta}})-\cR_2({f}^*) \right).$$
Then,  as $\cR_2({f})-\cR_2({f}^*) = \cR_2({f})-\cR_2({f}_{\tilde{\theta}})+
\cR_2({f}_{\tilde{\theta}})-\cR_2({f}^*)$, on the event $\cA$ and  conditionally on $\cD_n^{(1)}$, we deduce from the above inequality and Equation~\eqref{eq:eqDecompRisk4} that
\begin{equation*}
B_{{f}}  
 \leq C_K\left(D_{{f}}  + \Delta_n\right).
\end{equation*}
Hence,  from Inequality~\eqref{eq:eqBernstein}, we get for $t \geq \Delta_n$,
\begin{equation*}
\P\left( D_{{f}}-2\w{D}_{{f}} \geq t \right) \leq  \exp\left( -C_K n t\right),
\end{equation*}
which leads to  
\begin{equation*}
\P\left( \max_{\theta \in  {\cM_{\varepsilon}}} (D_{{f}}-2\w{D}_{{f}}) \geq t  \right) \leq |{\cM_{\varepsilon}}| \exp\left(-C_K n t\right).
\end{equation*}
In view of Equation~\eqref{eq:eqboundExpInt}, we then obtain that, conditionally
on $\cD_n^{(1)}$,
\begin{equation*}
 \mathbb{E}\left[\max_{ \theta \in {\cM_{\varepsilon}} } (D_{{f}_{\theta}}-2\w{D}_{{f}_{\theta}})\one_{\{\Omega\}} | \mathcal{D}_n^{(1)}\right]
 \leq \max\left(\Delta_n, \frac{C \log(|{\cM_{\varepsilon}}|)}{n}\right)
+\int_{C \log({\cM_{\varepsilon}})/n}^{+\infty} |{\cM_{\varepsilon}}| \exp(-Cnt)\dd t.
\end{equation*}
As before, we use that $\mathbb{E}\left[\Delta_n\right] \leq C/n$,  and we deduce from the above inequality by integrating over $\cD_n^{(1)}$ that
\begin{equation*}
\E\left[\max_{ \theta \in {\cM_{\varepsilon}} } (D_{{f}_{\theta}}-2\w{D}_{{f}_{\theta}})\one_{\{\Omega\}}\right] \leq  \frac{C \log(|{\cM_{\varepsilon}}|)}{n}.
\end{equation*}
Since for $\varepsilon= 1/(\log(n)n^3 M)$ we have that 
$\log(|\cM_{\varepsilon}|) \leq C(M+s^*) \log(nM)$, we obtain from the above inequality and Equation~\eqref{eq:bound2} that
\begin{equation*}
\E[ \cR_2(\w{f})- \cR_2({f}^*)] \leq C\dfrac{(M+s^*)\log(nM)}{n}.
\end{equation*}
From the above inequality, we get the desired  by applying the Zhang's lemma  
\begin{equation*}
 \E[ \cR(\hat{g})- \cR(g^*)] \leq  \dfrac{1}{\sqrt{2}}  \left(\E[ \cR_2(\w{f})- \cR_2({f}^*)]\right)^{1/2}.
\end{equation*}

\end{proof}